\documentclass[12pt]{amsart}
\usepackage{geometry}                % See geometry.pdf to learn the layout options. There are lots.
\usepackage{epstopdf}
\usepackage[usenames,dvipsnames]{color}
\usepackage{enumerate}
\usepackage{cite,pifont}
\usepackage{graphicx, amssymb, amsmath,amsthm}
\newcommand\abs[1]{\lvert#1\rvert}

\newcommand{\ds}{\displaystyle}

\newcommand{\mb}{\mathbb}
\newcommand{\mc}{\mathcal}

\newcommand{\ucite}[1]{\cite{#1}}

\def \a{\alpha} \def \b{\beta} \def \g{\gamma} \def \d{\delta}
\def \t{\theta} \def \T{\Theta}  \def \e{\epsilon}
\def \s{\sigma} \def \l{\lambda}  \def \o{\omega}
\def \O{\Omega} \def \D{\Delta}  \def \r{\rho}
\def \k{\kappa}  \def \G{\Gamma}
\def \di{\mathrm{dist}} \def \spa{\mathrm{span}}

\def\R{\mathbb{R}}
\def\vp{\varphi}
\def\tint{\text{int}}
\def \udi{\underline{\mathrm{dist}}}
\newtheorem{theorem}{Theorem}
\newtheorem{lemma}[theorem]{Lemma}
\newtheorem{cor}[theorem]{Corollary}
\theoremstyle{definition}

\newtheorem{definition}{Definition}
\newtheorem{remark}{Remark}

\newtheorem{prop}[theorem]{Proposition}

\title[Multiplicative Ergodic Theorem and Krein-Rutmann type Theorems]{On random linear dynamical systems in a Banach space. I. Multiplicative Ergodic Theorem and Krein-Rutmann type Theorems}
\author{}
\author[Lian]{Zeng Lian}
\address{Department of
Mathematical Science,
Loughborough University, United Kingdom}
\email{Z.Lian@lboro.ac.uk}
\author[Wang]{Yi Wang}
\address{Wu Wen-Tsun Key Laboratory, School of Mathematical Science, University of Science and Technology of China, Hefei, Anhui, 230026, People's Republic of China. Partially supported by NSF of China No.11371338, 11471305 and the Fundamental Research Funds for the Central Universities. }
\email{wangyi@ustc.edu.cn}
\begin{document}
\begin{abstract}

For linear random dynamical systems in a separable Banach space $X$, we derived a series of Krein-Rutman type Theorems with respect to co-invariant cone family with rank-$k$, which present a (quasi)-equivalence relation between the measurably co-invariant cone family and the measurably dominated splitting of  $X$. Moreover, such (quasi)-equivalence relation turns out to be an equivalence relation whenever (i) $k=1$; or (ii) in the frame of the Multiplicative Ergodic Theorem with certain Lyapunov exponent being greater than the negative infinity.
 For the second case, we thoroughly investigated the relations between the Lyapunov exponents, the co-invariant cone family and the measurably dominated splitting for linear random dynamical systems in $X$.
 \end{abstract}
\maketitle
%\section{}
%\subsection{}

\section{Introduction}

Lyapunov exponents \cite{Lya} play important roles in the study of the behavior of
dynamical systems. One of the most celebrated theorems on Lyapunov exponents is the so-called the Multiplicative Ergodic Theorem, which was first obtained by Oseledets \cite{O} in 1968.  This theorem plays a crucial role in the Pesin theory for describing the dynamics of nonuniformly hyperbolic diffeomorphisms on compact manifolds. During recent decades, many remarkable works and extensions of multiplicative ergodic theorem have been carried out from finite-dimensional systems to infinite-dimensional systems, and from deterministic dynamical systems to random dynamical systems (see \cite{LL,M,R,T} and the references therein).

In a separable Banach space $X$, the multiplicative Ergodic Theorem proposes an effective approach for decoupling $X$ into finitely (or infinitely)-many
random co-invariant subspaces which are strongly measurable and the ``angles'' between them are tempered. The existence of Lyapunov exponents with different values gives rise to exponential dichotomies, that is a well-known ``spectral-gap" condition (see e.g., \cite{ChLe,CLL,Pazy}), between the various invariant subspaces. Such gap has been attracting interest broadly in deterministic systems for a long time because of its close relationship with the theory of invariant manifolds. %(see \cite{}).

As was pointed out in \cite{M-PS,LLZ}, a genuinely essential condition for the existence of invariant manifolds is the existence of co-invariant cone family. Compared with the spectral gap condition, the invariant cones condition turns out to be more intuitive from the viewpoint of geometric description. While for the differential equations in question (see e.g., \cite{BJ,CFNT,M-PS,Tem}), the cones condition can be verified by showing, roughly speaking, that the vector field on the boundaries of the cones point inward.

We will show in Section \ref{S:TheoremA} that a family of finitely (or infinitely)-many nested invariant cones can be constructed under the assumptions of Multiplicative Ergodic Theorem in a separable Banach space $X$ (see Lian and Lu \cite{LL}).
To be more precisely, let us consider the linear cocycle over the metric dynamical system $(\O,\mc F,\mb P,\t)$ generated by a strongly measurable random variable $A:\O\to L(X)$,
\begin{equation}\label{T-n-expression}
T^n(\o):=A(\t^{n-1}\o)\cdots A(\o),\ n\in \mb N,
\end{equation}
 where $L(X)$ is the space of bounded linear operators from $X$ to itself; and moreover,  we assume that $A(\omega)$ is injective almost everywhere.  Define that
\begin{equation}\label{E:principle-LE}
\l_0(\o)=\lim_{n\to\infty}\frac1n\log\|T^n(\o)\|,
\end{equation}
\begin{equation}\label{E:ess-principle-LE}
 \k(\o)=\lim_{n\to\infty}\frac1n\log\|T^n(\o)\|_{\k},
 \end{equation}
 where $\|T^n(\o)\|$ is the norm of the operator $T^n(\o)$; and $\|T^n(\o)\|_\k$ is the Kuratowski measure (see \cite{Nuss}) of noncompactness of operator $T^n(\o)$. We then obtain the following (see also Theorem \ref{T:SplitImpliesConeSequence} in Section \ref{S:SettingsAndMainResults}):

 \vskip 3mm
\noindent {\bf Theorem A.}\label{T:SplitImpliesConeSequence-introduction}
{\it $\,$ Assume that $\log^+\|A\|\in \mc L^1(\O,\mc F,\mb P)$ and $\l_0>\k$ $\mb P$-a.e.. Then
there exist a $\t$-invariant subset $\tilde{\O}\subset \O$ of full measure, a $\t$-invariant function $k:\tilde{\O}\to \mb N\cup\{+\infty\}$ and a family of measurable cones $\{C_i(\cdot)\}_{1\le i<k(\cdot)+1}$ such that for each $\o\in \tilde{\O}$ and ${1\le i<k(\cdot)+1}$, %when $k(\cdot)>1$,
 \begin{itemize}%Conditions of Measurably Contracting
\item[(i)] $A(\o)C_i(\o)\subset C_i(\t\o);$

\item[(ii)] there is a tempered function $\chi_i:\O\to [1,+\infty)$ such that \begin{equation}\label{Strong-focusing-1}
\udi\left(A(\o)C_i(\o),X\setminus C_i(\t\o)\right)\ge \frac1{\chi_i(\o)},
     \end{equation}
     where "$\udi$" is the separation index defined in {\rm (\ref{E:SeparationIndex})}.
\end{itemize}
Moreover, one has $C_i(\o)\subset C_{i+1}(\o)$ and
 \begin{equation}\label{E:ConeI-1}%E:ConeI
 C_i(\o)=\left\{v\in X:\ |v-\pi^i(\o)v|_{\o,i}\le l_i(\o)|\pi^i(\o)v|_{\o,i}\right\}.
\end{equation}
Here $\pi^i$ is the projection onto $\bigoplus_{j=1}^i E_j(\o)$ associated with the splitting
 \[X=\Big(\bigoplus_{j=1}^i
E_j(\o)\Big)\oplus\Big(\big(\bigoplus_{j=i+1}^{k(\o)}E_j(\o)\big)\oplus
F(\o)\Big),\]
$l_i(\cdot)$ is a positive tempered function and $|\cdot|_{\o,i}$ is the Lyapunov norm defined in \eqref{E:LyapunovNorm}.
\vskip 3mm}

\noindent By virtue of Theorem A, it is clear that the spectral gap condition between the measurable invariant subspaces $E_j(\o),j=1,\cdots, k(\o),$ obtained in the multiplicative Ergodic Theorem, serves only as one of the accessible conditions to ensure the the existence of co-invariant cones.
\vskip 2mm

The present paper is motivated by the above insight on the co-invariant cone family.  We will thoroughly investigate the relations between  the invariant cone condition and measurably invariant splitting for linear random dynamical systems in a separable Banach space $X$, regardless of the situation whether multiplicative ergodic theorem is valid or not.

A geometrical insight of the structure of $C_i(\o)$ in \eqref{E:ConeI-1} will involve in a useful concept as a {\it cone of rank-$k$ (abbr. $k$-cone)}. Roughly speaking, a $k$-cone $C\subset X$ is a closed set that consists of straight lines and which contains a linear subspace of dimension $k$ and no linear subspace of higher dimension (see Definition \ref{D:k-cone}). To the best of our knowledge, $k$-cone was independently introduced  by Krasnoselskii et al. \cite{KLS} and Fusco and Oliva \cite{FO1}. In particular, given any traditional convex cone $K\subset X$, it is clear that $C=K\cup (-K)$ defines a cone of rank-$1$.

When $k=1$,  the cone invariance condition in Theorem A(i) simply yields that $A(\cdot)$ is a positive linear operator in $X$. Deterministic as well as random dynamical systems, generated by (family of) positive linear operators, have been widely studied because these systems provide relevant mathematical framework for the qualitative analysis of a variety of deterministic and random models via mappings and differential equations.

The celebrated Krein-Rutmann Theorem \cite{Deim,KR} (or Perron-Frobenius Theorem \cite{G} in finite-dimensions) had initiated research on the relation between the positivity (i.e., the cone invariance condition with $k=1$) and  the invariant splitting of $X$. It states that if an operator $A$ is compact and strongly positive with respect to a solid convex cone $K$ in a Banach space $X$, then there exists an invariant splitting $E\oplus F$ of $X$ with ${\rm dim}E=1$ such that $E\subset \{0\}\cup {\rm Int}(K\cup (-K))$, $F\cap (K\cup (-K))=\{0\}$, the spectral radius $r>0$ is the eigenvalues of $A$ corresponding to $E$, and all the other spectral points of $A$ satisfy $\abs{\l}<r$.

Krein-Rutmann Theorem had been extended to strongly positive and compact deterministic linear skew-product (semi)flows in finite-dimensional spaces \cite{Ru-2} and in Banach spaces \cite{PT2,Mier-1}. For random dynamical systems, Arnold et al. \cite{AGD} provided a version of random Perron-Frobenius Theorem for products of positive random matrices. Recently, Mierczy\'{n}ski and Shen \cite{MierSh-1} have established a random Krein-Rutmann Theorem. By utilizing the Multiplicative Ergodic Theorem developed in \cite{LL} for Banach space $X$, they \cite{MierSh-1} proved that $X$ possesses a measurably tempered invariant splitting $E(\cdot)\oplus F(\cdot)$ with ${\rm dim}E(\cdot)={\rm codim}F(\cdot)=1$ such that these measurable bundles are exponentially separated.  Moreover, many significant applications of Krein-Rutmann Theorem for deterministic as well as random dynamical systems can be found in \ucite{AGD,HuPo-1,HuPo-2,MP-N,MSh-book,MierSh-2,MierSh-3,MierSh-04,Pola,PT1}.

For a general solid $k$-cone $C\subset X$,  Fusco and Oliva \cite{FO1} established a generalized Perron-Frobenius theorem with respect to $C$ when ${\rm dim}X<\infty$, which states that if an operator $A$ strongly preserves $C$ (i.e.,  $A(C\setminus\{0\})\subset \tint C$), then $X$ admits an invariant splitting $E\oplus F$ with ${\rm dim}E=k$ such that $E\subset \{0\}\cup {\rm int}C$, $F\cap C=\{0\}$, and any spectral point $\l$ (resp. $\mu$) of $A$ restricted to $E$ (resp. $F$) satisfies $\abs{\mu}<\abs{\l}$. Independently, Krasnoselskii et al. \cite{KLS} obtained the generalized Krein-Rutman theory with respect to the $k$-cone $C$ in a Banach space.

Tere\v{s}\v{c}\'{a}k \cite{Te} first considered the deterministic linear skew-product semiflows which strongly preserves a solid $k$-cone $C$. In \cite{Te},  two complementary continuous invariant subbundles $E(\cdot)$ and $F(\cdot)$ were successfully constructed such that ${\rm dim}E(\cdot)={\rm codim}F(\cdot)=k$ and these bundles are exponentially separated (see in \cite{CLM-P} a different approach for scalar parabolic equations). Such exponentially separated continuous subbundles  has been successfully used in the study of transversality of invariant manifolds, structural stability and the dynamics for nonlinear scalar semiliear parabolic equations (see e.g., \cite{CCH,Joly-R,ShenYi-JDDE96,SWZ}).

Most recently, the present authors \cite{LW} investigated the linear random dynamical system $(\O,\mc F,\mb P,\t,T,X)$ in \eqref{T-n-expression}, with ${\rm dim}X<\infty$, which strongly preserves a solid $k$-cone $C$. %in the sense that $A(\o)(C\setminus\{0\})\subset \tint C$ for any $\o\in \O$.
Under some general assumptions, they showed that such system \eqref{T-n-expression} admits a tempered invariant splitting $X=E(\cdot)\oplus F(\cdot)$ with ${\rm dim}E(\cdot)={\rm codim}F(\cdot)=k$; and moreover, there exists a $\t$-invariant positive measurable function $\d:\O\to (0,\infty)$ such that
\begin{equation}\label{In:dominated-splitting}
\liminf_{n\to\infty} \frac1n\log \frac{|T^n(\o)u|}{|T^n(\o)v|}\ge\d(\o)
\end{equation}
for any $u\in E(\o)\setminus\{0\}$ and $v\in F(\o)\setminus\{0\}$. Here $\abs{\cdot}$ is the norm in $X$.
We hereafter call that $(\O,\mc F,\mb P,\t,T,X)$ possesses a {\it Measurably Dominated Splitting} if \eqref{In:dominated-splitting} is satisfied.

Motivated by Theorem A, we will focus on, in the present paper, the linear random dynamical system with a general $k$-cone $C$ in a separable Banach space $X$. A general $k$-cone $C$ is called {\it Measurably Contracting} with respect to system $(\O,\mc F,\mb P,\t,T,X)$ if $C$ satisfies (i)-(ii) of Theorem A.

We will thoroughly investigate the relations between measurably contracting $k$-cones, measurably dominated splitting and the multiplicative ergodic theorem. More precisely, among others, we will show the following (See also Theorems \ref{T:ConeImpliesSplit},\ref{T:SplitImpliesConeW} and \ref{T:SplitImpliesCone1d} in Section \ref{S:SettingsAndMainResults}):

\vskip 3mm
\noindent {\bf Theorem B.}\label{In:ConeImpliesSplit}%T:ConeImpliesSplit
{\it $\,$  Let $X$ be a separable Banach space $X$. Then for system $(\O,\mc F,\mb P,\t,T,X)$ defined in \eqref{T-n-expression}, we have
  \begin{itemize}
\item[(i)] Existence of the measurably contracting $k$-cone implies existence of the measurably dominated splitting of $X$ with $\dim E(\cdot)=k$.

 \item[(ii)] Existence of the measurably dominated splitting of $X$ with $\dim E(\cdot)=k$ implies
 existence of the eventually measurably contracting $k$-cone.

 \item[(iii)] When $k=1$, there exists a measurably contracting $1$-cone if and only if there exists a
 measurably dominated splitting $\dim E(\cdot)=1$.
    \end{itemize}
    }
\noindent Here the eventually measurably contracting property is defined in \eqref{E:-eventually measurable contracting}, which is weaker than the measurably contracting property.
\vskip 2mm

Our Theorem B(i) can be seen as {\it a random Krein-Rutmann Theorem with respect to general $k$-cones}, which extends all the above mentioned Krein-Rutmann type theorem either from deterministic systems to random systems, or from $k=1$ to $k>1$. We note that for random systems, due to the absence of compactness property of the base system $(\O,\mc F,\mb P,\t)$, one needs to develop new techniques to tackle the problem; while for $k>1$, the essential difference between $1$-cone $K\cup (-K)$ and $k$-cones is due to the lack of convexity, which makes the main difficulty in the investigation with respect to cones of high-rank.

Theorem B(i)-(ii) further presented a sort of (quasi)-equivalence relation between the measurably contracting $k$-cone and the measurably dominated splitting of $X$ with $\dim E(\cdot)=k$. In particular, when $k=1$, they are indeed equivalent to each other. Such (quasi)-equivalence will help one to understand the invariant cone condition from the viewpoint of the invariant splitting of the phase space, and vice versa, in the setting of the infinite-dimensional random dynamical systems. This will shed a light on investigating
 the theory of random inertial invariant manifolds for random dynamical systems in a Banach space.

 \vskip 1mm
It also deserves to point out that Theorem B is obtained without the integrability assumption $\log^+\|A\|\in \mc L^1(\O,\mc F,\mb P)$. If additionally, one imposes the integrability assumption as in the multiplicative ergodic theorem (\cite{LL}), then we have the following (see also a more general Theorem \ref{T:Cone&SplitMET} in Section \ref{S:SettingsAndMainResults}):

\vskip 3mm
\noindent {\bf Theorem C.}\label{In:ConeImpliesSplit-equivalence}%T:ConeImpliesSplit
{\it $\,$   Assume that $\log^+\|A\|\in \mc L^1(\O,\mc F,\mb P)$ in \eqref{T-n-expression} and $\l_0>-\infty$ $\mb P$-a.e., then the following conditions are equivalent:
  \begin{itemize}
\item[(a)] Existence of measurably contracting cone of nontrivial-rank (i.e., $k\ne 0$);

 \item[(b)] Existence of measurably dominated splitting of $X$;

  \item[(c)] $\l_0>\k$, $\mb P$-a.e..
    \end{itemize}
    }
\vskip 2mm
\noindent As shown in \cite{LL}, the multiplicative ergodic theorem proposed three types of possible situations of Lyapunov exponents being greater than $\k$: (i) None, (ii) finitely many or (iii) infinitely many, where case (ii) or case (iii) will happen if  $\l_{0}>\k$. However, there are no efficient methods in computing these quantities in general. So, it would be an interesting question that, under what condition there is at least one Lyapunov exponent which is greater than $\k$. Our Theorem C together with the Multiplicative Ergodic Theorem (\cite{LL}) implies that existence of
a measurably contracting $k$-cone will guarantee that the total multiplicity of Lyapunov exponents
greater than $\k$ is at least $1$ when $\k=-\infty$, and is at least $k$ when $\k>-\infty$.

The results obtained in this paper would have significant potential applications to the study of asymptotic dynamics of nonlinear random and stochastic partial differential equations.

The paper is organized as follows. In Section 2, we present our setting and state our main results. In Section 3, we state some known facts,  investigate and present several new properties of $k$-cones, which are crucial for the proof of the main results. In Section 4, we give the proof of Theorem B. In Section 5, we will investigate the relation between the eventually measurably contracting property and the measurably dominated splitting in Probability. Based on this, we will prove Theorem C in Section 6. Moreover, Theorem A will be also proved in Section 6. In the Appendix, we collect several well known results, and also present some technical results which is useful to our main results.

\section{Settings and Main Results}\label{S:SettingsAndMainResults}%S:SettingsAndMainResults
\subsection{Settings}\label{S:Settings}%S:Settings
In this subsection, we introduce fundamental notion and basic settings of the problem.
\subsubsection{Random dynamical system}\label{S:RandomDynamicalSystem}%S:RandomDynamicalSystem
Let $(X,\abs{\cdot})$ be a separable Banach space with Borel $\sigma$-algebra $\mc B$. Let $(\Omega, \mc F, \mb P)$ be a complete probability space and let $\theta:\Omega\to\Omega$ be a bijection such that both $\theta$ and $\theta^{-1}$ are $\mb P$-invariant measurable maps.
%Without losing any generality, we assume that $\theta$ is ergodic.
A random dynamical system on measurable space $(X, \mc B)$ over the so called metric dynamical system $(\Omega,\mc F,\mb P,\theta)$ is a family $\{\varphi(\omega,t):t\in\mb T\}$, where $\mb T=\mb Z,\mb Z^+,\mb R, \text{ or } \mb R^+$, of measurable transformations on $X$ satisfying the following  cocycle properties: For $\mb P$-almost all $\omega\in \Omega$
\begin{itemize}
\item $\varphi(\omega,0)=id$;
\item $\varphi(\omega.s+t)=\varphi(\theta^s\omega,t)\circ \varphi(\omega,s)$ for all $s,t\in\mb T$.
\end{itemize}

In this paper, we only consider the one-side discrete time systems, i.e., $\mb T= \mb Z^+$, and denote that
$$\vp(\o)=\vp(\o,1), \text{ and } \vp(\o,n)=\vp(\t^{n-1}\o)\circ\cdots\circ \vp(\o),n\in\mb N.$$

A $\vp$-invariant measure is a probability measure $\mu$ on $\O\times X$ with marginal $\mb P$ on $\O$ which is invariant under the induced skew product map
$$\T: \O\times X\to \O\times X,\  (\o,x)\to (\t\o,\vp(\o)x).$$

In this paper, we study the linear cocycle over the metric dynamical system $(\O,\mc F,\mb P,\t)$ generated by the random variable $A:\O\to L(X)$,
\begin{equation}\label{E:LinearCocycle}
T^n(\o):=A(\t^{n-1}\o)\cdots A(\o),\ n\in \mb N,
\end{equation}
 where $L(X)$ is the space of bounded linear operators from $X$ to itself. We will always assume $A(\o)$ is injective almost everywhere.
 \vskip 3mm

\subsubsection{k-cones}\label{S:k-cones}%S:k-cones
 In this subsection, we define $k$-cones and some corresponding notions.
 \begin{definition}\label{D:k-cone}
 A closed subset $C$ of $ X$ is called a $k$-dimensional closed cone of $X$ if the following are satisfied:
 \begin{itemize}
 \item[i)] For any $v\in C$ and $l\in \mb R$, $lv\in C$;
 \item[ii)] There exists a $k$-dimensional subspace $K\subset X$ such that $K\subset C$;
 \item[iii)] There exists a $k$-codimensional subspace $L\subset X$ such that $L\cap C=\{0\}$.
 \end{itemize}
 \end{definition}
A $k$-cone $C$ is called {\em solid} if $\tint C$ is not empty.

 \subsection{Main Results}\label{S:MainResults}%S:MainResults
% At first, we define the system of cones we will work on. Over an ergodic invertible metric dynamical system $(\O,\mc F,\mb P,\t)$, any strongly measurable operator-value function $A:\O\to L(X)$ generates a random linear cocycle $\{T^n(\o)\}_{n\ge0}$.

For the system $(\O,\mc F,\mb P,\t,T,X)$ defined in Section \ref{S:RandomDynamicalSystem}, we consider a solid cone system $\mc C:=\{C(\o)\}_{\o\in\O}$. We call that $\mc C$ is {\it Measurably Contracting } with respect to system $(\O,\mc F,\mb P,\t,T,X)$ if the following are satisfied:
 \begin{itemize}%Conditions of Measurably Contracting
 \item[{\bf C1})] There exists a measurable $\t$-invariant function $k:\O\to \mb N$ such that for $\mb P$-a.e. $\o$, $C(\o)$ is a closed solid $k(\o)$-cone and
 $$A(\o)C(\o)\subset C(\t\o).$$
 \item[{\bf C2})] There exists a measurable splitting $X=E_0(\o)\oplus F_0(\o)$ such that $\dim E_0(\o)=\text{codim} F_0(\o)=k(\o)$, and  $E_0(\o)\subset {\rm int }C(\o)\cup\{0\},\ F_0(\o)\cap C(\o)=\{0\}$, for $\mb P$-a.e.;
 \item[{\bf C3})] There exists a measurable function $\chi:\O\to [1,+\infty)$  such that for $\mb P$-a.e.
 $$\lim_{n\to\pm\infty}\frac1n \log \chi(\t^n\o)=0$$ and
 $$\udi\left(A(\o)C(\o),X\setminus C(\t\o)\right)\ge \frac1{\chi(\o)}.$$
 \end{itemize}
Here "$\udi$" is defined as in (\ref{E:SeparationIndex}).\\

\begin{remark}\label{R:SNormal}%R:SNormal
Without losing any generality, we can additionally assume that\\
\begin{itemize}
\item[{\bf C4})]  $\,\,$$\udi(F_0(\o),C(\o))>\frac1{\chi(\o)}$, for $\mb P$-a.e. $\o$.
\end{itemize}
\medskip
Actually, for a given cone family $\mc C$ satisfying C1)-C3), one can construct another cone family $\mc C'$ which satisfies C1)-C4) by choosing a proper $\chi$. For sake of convenience, we will assume C4) when proving the theorems, and the construction of $\mc C'$ is left to appendix.
\end{remark}

%%%%%%%%%%%%%%%%%%%%%%%%%%%%%%%%%%%%%%%%  Eventually Measurably Contracting  %%%%%%%%%%%%%%%%%%%%%%%%%%%%%%%%%%%%%%%%%%%%%%%%%
A weak version of the measurably contracting property is as follows. We say a cone family $\mc C$ is {\it eventually measurable contracting} with respect to system $(\O,\mc F,\mb P,\t,T,X)$ if C1), C2) and the following is satisfied:\\

\noindent {\bf C3')} There exists measurable functions $\chi:\O\to [1,+\infty)$ and $N:\O\to \mb N$ such that for $\mb P$-a.e.
 $$\lim_{n\to\pm}\frac1n \log \chi(\t^n\o)=0$$ and for any $m\ge N(\o)$
 \begin{equation}\label{E:-eventually measurable contracting}
 \udi\left(T^m(\o)C(\o),X\setminus C(\t^m\o)\right)\ge \frac1{\chi(\t^m\o)}.
 \end{equation}

 %%%%%%%%%%%%%%%%%%%%%%%%%%%%%%%%%%%%%%%%%%%%%%%%%%%%%%%%%%%%%%%%%%%%%%%%%%%%%%%%%%%%%%%%%%%%%
\vskip 3mm
We say the system $(\O,\mc F,\mb P,\t,T,X)$ admits a {\it Measurably Dominated Splitting} if
there is a $\t$-invariant set $\tilde \O\subset \O$ of full $\mb P$-measure and a  $\t$-invariant measurable function  $k:\O\to \mb N$ such that the following are satisfied:
 \begin{itemize}
 \item[{\bf D1)}] For all $\o\in\tilde \O$, there exists a invariant splitting
 $$X=E(\o)\oplus F(\o),$$
 with $\dim E(\o)=\text{codim} F(\o)=k(\o)$, where, by invariance,  we mean that  $$A(\o)E(\o)=E(\t\o)\text{ and }A(\o)F(\o)\subset F(\t\o);$$
 \item[{\bf D2)}] Letting $\pi_E(\o),\pi_F(\o)$ be the projections associated to the splitting $X=E(\o)\oplus F(\o)$, then $\pi_{E},\pi_{F}$ are strongly measurable, and
 $$\lim_{n\to\pm\infty}\frac1n\log\underline\di(E(\t^n\o),F(\t^n\o))=0,$$
 $$i.e. \lim_{n\to\pm\infty}\frac1n\log\max\{\|\pi_E(\t^n\o)\|,\|\pi_F(\t^n\o)\|\}=0;$$
\item[{\bf D3)}] There exists a $\t$-invariant measurable function $\d:\O\to (0,\infty)$, and a measurable function $K(\cdot):\O\to[1,\infty)$ such that
$$\lim_{n\to\pm\infty}\frac1n\log K(\t^n\o)=0;$$
and moreover,  for all $\o\in\tilde \O$ and $n\ge 0$,
    \begin{equation}\label{E:UniformSeparation}%E:UniformSeparation
 \left\|\left(T^n(\o)|_{E(\o)}\right)^{-1}\right\|\left\|T^n(\o)|_{F(\o)}\right\|\le K(\o) e^{-n\d(\o)}.
 \end{equation}

  \end{itemize}
 \begin{remark}\label{R:UniformSeparation}%R:UniformSeparation
It is easy to see that D3) implies that for any $u\in E(\o)\setminus\{0\}$ and $v\in F(\o)\setminus\{0\}$
 $$\liminf_{n\to\infty} \frac1n\log \frac{|T^n(\o)u|}{|T^n(\o)v|}\ge\d(\o).$$
 \end{remark}
\vskip 3mm
%%%%%%%%%%%%%%%%%%%%%%%%%%%%%%%%% Dominated Splitting in Probability %%%%%%%%%%%%%%%%%%%%%%%%%%%
We say the system $(\O,\mc F,\mb P,\t,T,X)$ has an {\it Measurably Dominated Splitting in Probability} if D3) is replaced by the following:\\

\noindent {\bf D3')}  There exist measurable functions $\d:\O\to (0,\infty)$ and $K:\O\to[1,\infty)$  such that
 for all $\o\in\tilde \O$ there exist a sequence of integers $\{n_j(\o)\}_{j\in \mb Z}$ such that
 $$n_{0}(\o)=0,\ n_{j}(\o)<n_{j+1}(\o),\ \liminf_{j\to\infty}\frac{j}{n_j(\o)}>0,$$
 $$\d(\t^{n_{j}(\o)})=\d(\o),\ \lim_{j\to\pm\infty}\frac1j\log K(\t^{n_{j}(\o)}\o)=0,$$
 and for $j\ge0$
    \begin{equation}\label{E:NonUniformSeparation}%E:NonUniformSeparation
 \left\|\left(T^{n_j(\o)}(\o)|_{E(\o)}\right)^{-1}\right\|\left\|T^{n_j(\o)}(\o)|_{F(\o)}\right\|\le K(\o) e^{-j\d(\o)}.
 \end{equation}
\vskip 3mm
%%%%%%%%%%%%%%%%%%%%%%%%%%%%%%%%%%%%%%%%%%%%%%%%%%%%%%%%%%%%%%%%%%%%%%%%%%%%%%%%%%%%%%%
Our main results are as follows:

 \begin{theorem}\label{T:ConeImpliesSplit}%T:ConeImpliesSplit
 Existence of measurably contracting $k$-cone family $\mc C$ implies existence of measurably dominated splitting with $\dim E(\cdot)=k(\cdot)$ and $F(\cdot)\cap C(\cdot)=\{0\}$.
 \end{theorem}

 %%%%%%%%%%%%%%%%%% %%%%%%%%%%%%%%%%%%%%%%%%%%%%%%%%%%%%%%%%%%%%%%%%%%%%%%%%%%%%%%%%
 \begin{theorem}\label{T:SplitImpliesConeW}%T:SplitImpliesConeW
 Existence of measurably dominated splitting with $\dim E(\cdot)=k(\cdot)$ implies existence of eventually measurably contracting $k(\cdot)$-cone family $\mc C$.
 \end{theorem}

 \begin{theorem}\label{T:ConeImpliesSplitW}%T:ConeImpliesSplitW
 Existence of eventually measurably contracting $k$-cone family $\mc C$ implies existence of measurably dominated splitting in probability with $\dim E(\cdot)=k(\cdot)$ and $F(\cdot)\cap C(\cdot)=\{0\}$.
 \end{theorem}

\vskip 2mm

 %%%%%%%%%%%%%%%%%%%%%%%%%%%%%%%%%%%%%%%%%%%%%%%%%%%%%%%%%%%%%%%%%%%%%%%%%%%%%%%%%%%%%%%%%%%%%%%%%%%%%%%%%%%%%%%%%%%%%%%%%%%%%%%%%%
 In general, existence of measurably contracting cone family is a sufficient condition for existence of measurably dominated splitting. However, starting with measurably dominated splitting, if $\dim E=1$, one can construct a measurably contracting cone family.

 \begin{theorem}\label{T:SplitImpliesCone1d}%C:SplitImpliesCone1d
 There exists measurably contracting $1$-cone family $\mc C$ if and only if there exists measurably dominated splitting with $\dim E(\o)=1$ $\mb P$-a.e..
 \end{theorem}

By virtue of Theorems \ref{T:ConeImpliesSplit}-\ref{T:ConeImpliesSplitW}, existence of measurably dominated splitting $E(\cdot)\oplus F(\cdot)$ in probability (with $\dim E(\cdot)=k(\cdot)$) is the weakest one compared with other properties. The following Theorem indicates that if asymptotic expanding rates of $T^n|_E$ is well controlled, the one can also construct a measurably contracting $k(\cdot)$-cone family.

\begin{theorem}\label{T:SplitImpliesConeL1}%T:SplitImpliesConeL1
 Assume additionally that

 {\bf D4)} $\log^+\|A\|\|(A|_{E})^{-1}\|\in \mc L^1(\O,\mc F, \mb P)$.

\noindent Then existence of measurably dominated splitting {\em in probability} with $\dim E(\cdot)=k(\cdot)$ implies existence of measurably contracting $k$-cone family $\mc C$ .
 \end{theorem}

  In the following, we will investigate the relation between the measurably contracting cone family, measurably dominated splitting and the Lyapunov exponents given in
 the multiplicative ergodic theorem (see \cite{LL}). Let $\l_0$ and $\k$ be defined in \eqref{E:principle-LE} and \eqref{E:ess-principle-LE}.

   \begin{theorem}\label{T:SplitImpliesConeSequence}%T:T:SplitImpliesConeSequence
Assume that $\log^+\|A\|\in \mc L^1(\O,\mc F,\mb P)$ and $\l_0>\k$ $\mb P$-a.e.. Then
there exist $\t$-invariant functions $k:\O\to \mb N\cup{+\infty}$, and $\{k_i:\O\to \mb N\}_{1\le i< k+1}$ with $k_{i}<k_{i+1}$, and a family of measurably contracting $k(\cdot)$-cones  $\{C_i(\cdot)\}_{1\le i<k(\cdot)+1}$  such that $C_i(\cdot)$ are $k_i$-cones; and when $k(\cdot)>1$, one has
$$C_i(\cdot)\subset C_{i+1}(\cdot),\text{ for }\ 1\le i< k(\cdot)+1,\ \mb P-a.e..$$
\end{theorem}
\vskip 3mm

The following equivalence Theorem is a general version of Theorem C.
\begin{theorem}\label{T:Cone&SplitMET}%T:Cone&SplitMET
Assume that $\log^+\|A\|\in \mc L^1(\O,\mc F,\mb P)$ and $\l_0>-\infty$ $\mb P$-a.e.. Then the following conditions are equivalent:
\begin{itemize}
\item[a)] Existence of measurably contracting cone;
\item[b)] Existence of measurably dominated splitting;
\item[c)] Existence of measurably eventually contracting cone;
\item[d)] Existence of measurably dominated splitting in probability;
\item[e)] $\l_{0}>\k$, $P$-a.e..
\end{itemize}
\end{theorem}

%%%%%%%%%%%%%%%%%%%%%%%%%%%% NEW SECTION %%%%%%%%%%%%%%%%%%%%%%%%%%%%%%%%%

\section{Properties of k-cones}\label{S:PropertiesOfk-cones}%S:PropertiesOfk-cones
In this section, we state some properties of $k$-cones proved in  \cite{KLS}, \cite{LW}, and \cite{Te}, and also generate some new properties used in this paper.\\

For a solid $k$-cone $C\subset X$, we denote by $\G_k(C)$ the set of $k$-dimensional subspaces inside $\tint C$, which is a metric space by endowing {\it the gap metric} (see, e.g. \cite{Kato,LL}).  For sake of completeness, we give the definition of the gap distance here. For any nontrivial closed subspaces $E,F\subset X$, define that
 \begin{equation}\label{E:GapDistance}%E:GapDistance
 d(E,F)=\max\left\{\sup_{v\in E\cap S}\inf_{u\in F\cap S}|v-u|, \sup_{v\in F\cap S}\inf_{u\in E\cap S}|v-u|\right\},
 \end{equation}
 where $S=\{v\in X|\ |v|=1\}$ is the unit ball. It is clear to see that $d(E,F)\le 2$ for any nontrivial closed subspace $E,F\subset X$; and moreover, we set $d(\{0\},\{0\})=0$, $d(\{0\},E)=2$ for any $E\neq\{0\}$.

  Now we introduce a new index $\a_C$ on $\G_k(C)$, which we call {\it the angle-index} of the elements in $\G_k(C)$ following Tere\v{s}\v{c}\'{a}k \cite{Te}. In this section, we will investigate the relation between the angle index $\a_C$ and the gap metric $d$. For any $u,v\in  \text{int} C$, define
 $$\a_0^C(u,v)=\inf\{\a\ge0:\b v-u\in\tint C \text{ for all }\b\ge \a\},$$ and the angle
  \begin{equation}\label{D:Angle}%D:Angle
 \a_C(u,v)=\a_0^C(u,v)\a^C_{0}(v,u).
 \end{equation}
Clearly, $\a_C(u,v)=\a_C(v,u)$; and moreover, $\a_C(u,v)=0$ if $u,v$ generate a plane contained in $C$. For any $E_1, E_2\in \G_k(C)$, define the angle-index of $E_1,E_2$ as
 \begin{equation}\label{D:AngleE} %D:AngleE
 \a_C(E_1,E_2)=\sup\{\a_{C}(u,v)|\ u\in E_1-\{0\}, v\in E_2-\{0\}\}.
 \end{equation}
 The following lemma summarizes the properties of $\a_C$, which is adopted from \cite{Te}.
 \begin{lemma}\label{L:AngleProperties}%L:AngleProperties
 For any $u,v\in \text{int} C$ and $E_1,E_2\in \G_k(C)$, the following are true:
 \begin{itemize}
 \item[i)] $\a_C(u,v)$ is upper semicontinuous on $u,v$ i.e. for vectors $u_i,v_i\in\text{int} C$ converging to $u,v\in\text{int} C$ respectively, $\limsup_{i\to\infty}\a_C(u_i,v_i)\le \a_C(u,v)$;
 \item[ii)] $\a_C(E_1,E_2)=0\ \iff\ E_1=E_2$;
 \item[iii)] $\a_C(E_1,E_2)$ is upper semicontinuous on $E_1,E_2$ in the sense that for $E_{1i},E_{2i}$ converging to $E_1,E_2\in\G_{k}(C)$ respectively, when $i\to\infty$ and $\a_C(E_{1i}, E_{2i})$ is well defined for all $i$ then
 $$\limsup_{i\to\infty}\a_C(E_{1i},E_{2i})\le \a_C(E_1,E_2);$$
 \item[iv)] For any $A\in\mc S(C)$,
  $$\a_{C}(Au,Av)<\a_{C}(u,v),\ \forall u,v\in\tint C.$$
 \end{itemize}
 \end{lemma}
\begin{proof}
See \cite[Proposition 4.1]{Te}.
\end{proof}
It deserves to point out that the angle-index $\a_C$ can be seen as a sort of higher-dimensional analog of the Hilbert projective metric on unit vectors in the interior of a positive convex
cone (i.e., $k$-cone with $k=1$). However, while the Hilbert projective metric is a continuous function of those vectors, $\a_C(u,v)$ is only upper semi-continuously dependent on $u,v$, because of the non-convexity of the $k$-cone $C$ in general.

In order to investigate the relation between $\a_C$ and $d$, we introduce

\begin{definition}\label{D:Separation-Index-L-C}
For any nontrivial linear subspace $L\subset X$ and nonempty subset $C\subset X$, the {\it separation index} between $L$ and $C$ is defined by
\begin{equation}\label{E:SeparationIndex}%E:SeparationIndex
 \underline{\di}(L,C)=\inf_{v\in L\cap S}\{\inf_{u\in C}|v-u|\}.
 \end{equation}
\end{definition}

The following two technical lemmas (Lemma \ref{L:ProjectionNorm} and \ref{L:ProjectiveMetric}) have been proved by the authors in \cite{LW} when $X$ is a finite dimensional space. In fact, the exactly same proofs are still valid when $X$ is a Banach space. Therefore, we omit the proofs here and refer the reader to \cite{LW} for details. Furthermore, Lemma \ref{L:ProjectiveMetric} makes explicitly the relation between $\a_C$ and $d$ with the aid of the separation index ``$\underline{\di}$".
\begin{lemma}\label{L:ProjectionNorm}%L:ProjectionNorm
 Let $X=V\oplus L$ be a splitting where $V,L$ are closed subspace. Then
 $$\underline{\di}(L,V):=\inf_{v\in L\cap S_{X}}\{\inf_{u\in V}|v-u|\}=\frac1{\|\pi_{L}\|}.$$
 \end{lemma}

 \begin{lemma}\label{L:ProjectiveMetric}%L:ProjectiveMetric
 Let $C$ be a closed solid cone, then the following are true:
 \begin{itemize}
 \item[I)] For any $u_1,u_2\in \tint C$, $\a_C(u_1,u_2)< 1$ implies that $\a_C(u_1,u_2)=0$. Further more, for any $E_1,E_2\in \G_k(C)$, $\a_C(E_1,E_2)\le 1$ implies that $\a_C(E_1,E_2)=0$, thus $E_1=E_2$;
 \item[II)] For any $E_{1},E_{2}\in\G_{k}(C)$, let $\d=\min\{\underline{\di}(E_{i},X\setminus C),i=1,2\}$,
 then there exists $K>0$ which depends on $\d$ only, such that
$$\log\a_C(E_{1},E_{2})\le\a_C(E_1,E_2)-1\le K d(E_1,E_2).$$
Furthermore, it suffices to take $K=2\left(\frac1{\d^{4}}+\frac6{\d^{3}}+\frac7{\d^{2}}+\frac6\d+2\right)$ which is a polynomial of $\frac1\d$;

\item[III)] If we further assume that there exists a $k$-codimensional subspace $L$ such that
 $\underline{\di}(L,C)>0,$  then for any $E_1,E_2\in \G_k(C)$,
\begin{align*}
 d(E_1,E_2)&\le \min\left\{\frac2{\underline{\di}(L,C)}(\max\{\a_C(E_1,E_2),1\}-1),2\right\}\\
&\le\min\left\{\frac2{\underline{\di}(L,C)}\frac{\ln (\max\{\a_C(E_1,E_2),1\})}{\ln 2},2\right\}.
 \end{align*}
 \end{itemize}
 \end{lemma}
 \vskip 1mm

As we aware so far, the following notions are generalizations of the ones first introduced by Kranosel'skij et al. in \cite{KLS}, which play the essential roll in deriving the invariant subspaces.
\vskip 2mm

Given any solid $k$-cones $C_1,C_2$ in $X$, we define the set of operators which maps $C_1$ into $C_2$ as follows
 $$\tilde  {\mc S}(C_1,C_2)=\{A\in L(X)|\ A \text{ is injective and } A(C_1\setminus\{0\})\subset \tint C_2\}.$$

\begin{definition}\label{D:Focusing}%D:Focusing
 An operator $A\in\tilde{\mc S}(C_1,C_2)$ for closed solid $k$-cones $C_1$ and $C_2$ is called {\em focusing on $C_1$ with respect to } $C_2$ if
 $$\chi(A|_{C_1},C_2):=\sup\{\a_{C_2}(A(u),A(v))|\ u,v\in C_1, Au,Av\neq 0\}<\infty.$$
 \end{definition}
\noindent  In the following, we denote $\mc S_{F}(C_1,C_2)$ the set of operators which are injective and focusing on $C_1$ with respect to $C_2$. By definition, it is clear that $\mc S_{F}(C_1,C_2)\subset \tilde{\mc S}(C_1,C_2)$.

   \begin{definition}\label{D:Normal}%D:Normal
 A closed solid $k$-cone $C$ is called {\em normal} if there is a number $b>0$ such that for any $E\in \G_{k}(C)$ and any $u\in C\cap S$
 $$\inf\{|u-v||\ v\in E, |v|\le 1\} \le b\log\a_C(u,E)\le b (\a_C(u,E)-1).$$
\end{definition}

A quick consequence of Lemma \ref{L:ProjectiveMetric} is the following corollary, which concerns on
the relation between the separation index ``$\underline{\di}$" and focusing property and normality.

    \begin{cor}\label{C:NormalityAndFocusing}%C:NormalityAndFocusing
    Let $C$, $C_1$ and $C_2$ be closed solid $k$-cones, then
    \begin{itemize}
    \item[a)]  For any $A\in \tilde{\mc S}(C_1,C_2)$, if there is a number $\d>0$ such that
    \begin{equation}\label{SF-proper}
    \underline{\di} (A(C_1),X\setminus C_2)\ge \d,
    \end{equation}
    then $$A\in S_{F}(C_1,C_2) \text{ with } \chi(A|_{C_1},C_2)\le p(\frac{1}{\d}),$$ where $p(x)=4(x^{4}+6x^{3}+7x^{2}+6x+2)$.
    \item[b)] If there exists a $k$-codimensional subspace $L$ such that
    \begin{equation}\label{SN-proper}
    \underline\di(L,C)>0,
     \end{equation}
    then $C$ is normal and one can take $b=\dfrac4{\underline{\di} (L,C)},$
    where $b$ is the one in the definition of normality.
    \end{itemize}
    \end{cor}
 In the rest of the paper, we call these assumptions \eqref{SF-proper} and \eqref{SN-proper} in Corollary \ref{C:NormalityAndFocusing} {\em strong normality} and {\em strong focusing}, respectively; and denote $\mc S_{SF}(C_1,C_2)$ the collection of all strong focusing operator on $C_1$ with respect to $C_2$, and define the strongly focusing number for a given $A\in\mc S_{SF}(C_1,C_2)$
 \begin{equation}\label{E:FocusingNumber}%E:FocusingNumber
 \chi_{S}(A|_{C_1},C_2):=\frac1{\underline\di(A(C_1),X\setminus C_2)}\quad(\ge 1).
 \end{equation}

\vskip 2mm
 \begin{remark}\label{R:SomeQuestions}%R:SomeQuestions
The difference between $1$-cone and $k$-cone, finite dimensional case and infinite dimensional case, are quite significant:
 \begin{itemize}
 \item[a)] Unlike the 1-cone case, $\log^+\a_C$ or $\a_C-1$ may not define a metric on $\G_k(C)$, because of the non-convexity of $C$;
 \item[b)] For finite dimensional case, normality and focusing follow the compactness of $C\cap S$ and $L\cap S$ directly. But for infinite dimensional case, such assumption can not be satisfied in general (from the definition of the cone).
        \end{itemize}
    \end{remark}

Before ending this section, we need to built certain contraction property of $A(\o)$ on $C(\o)$ with respect to the angle index $\a_C$. Strategy used in \cite{KLS} will be borrowed here.
As an analog of the so called {\it Birkhoff's contraction coefficient}, we define, for any closed solid $k$-cones $C_1$ and $C_2$, and $A\in \mc S_{F}(C_1,C_2)$,
\begin{align}\begin{split}\label{E:Birkhoff'sContractionCoefficient }%E:Birkhoff's contraction coefficient
&\tau_{V}(A):=\sup_{v,u\in \tint C_1,\a_{C_1}(u,v)> 1}\frac{\a_{C_2}(Av,Au)-1}{\a_{C_1}(u,v)-1},\\
&\tau(A):=\sup_{E,F\in\G_k(C_1),E\neq F}\frac{ \a_{C_2}(A(E),A(F))-1}{\a_{C_1}(E,F)-1}.
\end{split}\end{align}
By definition, $\tau$ and $\tau_V$ are sub-multiplicative, by which we mean that for any closed solid $k$-cones $C_1$, $C_2$ and $C_3$, and $A\in \mc S_{F}(C_1,C_2)$ and $B\in \mc S_{F}(C_2,C_3)$,
\begin{equation}\label{E:SubMult}%E:SubMult
\tau(BA)\le \tau(B)\tau(A)\,\text{ and }\,\tau_V(BA)\le \tau_V(B)\tau_V(A).
\end{equation}
\begin{lemma}\label{L:ContractionOnCone0}%L:ContractionOnCone0
Let $C_1$ and $C_2$ be closed solid $k$-cone, and $A\in\mc S_F(C_1,C_2)$, then
$$0\le \tau(A)\le \tau_{V}(A)\le \frac{\chi(A|_{C_1},C_2)-1}{\chi(A|_{C_1},C_2)+1}.$$
\end{lemma}

\begin{proof}
Clearly, $\tau(A)\le \tau_{V}(A)$. We will show that for any $u,v\in C_1-\{0\}$
\begin{equation}\label{E:ContractionOnCone0-1}%E:ContractionOnCone0-1
\max\{\a_{C_2}(Au,Av)-1,0\}\le \frac{\chi(A|_{C_1},C_2)-1}{\chi(A|_{C_1},C_2)+1}(\a_{C_1}(u,v)-1).
\end{equation}
By Lemma \ref{L:AngleProperties}(iv), (\ref{E:ContractionOnCone0-1}) is obvious if $\a_{C_1}(u,v)\le 1$. So we only consider the case that $\a_{C_1}(u,v)>1$, while $\a_0^{C_1}(u,v),\a_0^{C_1}(v,u)\in (0,\infty)$ and $\a_0^{C_1}(v,u)>\frac1{\a_0^{C_1}(u,v)}$. Let $w_1=\frac1{\a_0^{C_1}(u,v)}u-v$ and $w_2=v-\a_0^{C_1}(v,u)u$. We may assume $Aw_1,Aw_2\neq 0$, otherwise (\ref{E:ContractionOnCone0-1}) is obvious.\\

First, we consider the case that $\a_{C_2}(Aw_1,Aw_2)<1$. Then, by 1) of Lemma \ref{L:ProjectiveMetric}, $sAw_1-Aw_2\in \tint C_2$ for all $s\ge 0$, which implies that
$$tAu-Av\in\tint C_2,\text{ for all }t=\frac{\frac{s}{\a_0^{C_1}(u,v)}+\a_0^{C_1}(v,u)}{s+1},\ s\ge 0.$$
So $tAu-Av\in\tint C_2$ for all $t\in\left(\frac1{\a_0^{C_1}(u,v)},\a_0^{C_1}(v,u)\right)$. On the other hand, by definition of $\a_0^{C_1}(v,u)$ and $\a_0^{C_1}(u,v)$, and since $A\in \mc S_F(C_1,C_2)$, we have that
$$tAu-Av\in \tint C_2 \text{ for all }t\in\left[0,\frac1 {\a_0^{C_1}(u,v)}\right]\cup [\a_0^{C_1}(v,u),\infty).$$
Thus $\a_{C_2}(Au,Av)=0$, and hence (\ref{E:ContractionOnCone0-1}) holds.\\

Now we consider the case $\a_{C_2}(Aw_1,Aw_2)>1$. Since $\a_{C_2}(Aw_1,Aw_2)\le \chi(A|_{C_1},C_2)$, we have that $\a_0^{C_2}(Aw_2,Aw_1)\le \frac{\chi(A|_{C_1},C_2)}{\a_0^{C_2}(Aw_1,Aw_2)}$. Then $sAw_1-Aw_2\in \tint C_2$ for all $s\in \left(0,\frac1{\a_0^{C_2}(Aw_1,Aw_2)}\right)\cup \left(\frac{\chi(A|_{C_1},C_2)}{\a_0^{C_2}(Aw_1,Aw_2)},\infty\right)$, equivalently, for these $s$,
$$\frac{\frac{s}{\a_0^{C_1}(u,v)}+\a_0^{C_1}(v,u)}{s+1}Au-Av\in \tint C_2.$$
Also note that $tAu-Av=A(tu-v)\in \tint C_2$ for all $t\in \left[0,\frac1{\a_0^{C_1}(u,v)}\right]\cup [\a_0^{C_1}(v,u),\infty)$. Then we have that
\begin{align*}
&\a_{C_2}(Au,Av)-1=\a_0^{C_2}(Av,Au)\a_0^{C_2}(Au,Av)-1\\
\le& \frac{\frac{1}{\a_0^{C_1}(u,v)\a_0^{C_2}(Aw_1,Aw_2)}+\a_0^{C_1}(v,u)}{\frac1{\a_0^{C_2}(Aw_1,Aw_2)}+1}  \frac{\frac{\chi(A|_{C_1},C_2)}{\a_0^{C_2}(Aw_1,Aw_2)}+1}
{\frac{\chi(A|_{C_1},C_2)}{\a_0^{C_2}(Aw_1,Aw_2)\a_0^{C_1}(u,v)}+\a_0^{C_1}(v,u)}-1\\
=&\frac{\frac{1}{\a_0^{C_2}(Aw_1,Aw_2)}+\a_{C_1}(u,v)}{\frac1{\a_0^{C_2}(Aw_1,Aw_2)}+1}  \frac{\frac{\chi(A|_{C_1},C_2)}{\a_0^{C_2}(Aw_1,Aw_2)}+1}
{\frac{\chi(A|_{C_1},C_2)}{\a_0^{C_2}(Aw_1,Aw_2)}+\a_{C_1}(u,v)}-1\\
=&\frac{(\a_{C_1}(u,v)-1)(\chi(A|_{C_1},C_2)-1)\a_0^{C_2}(Aw_1,Aw_2)}{(1+\a_0^{C_2}(Aw_1,Aw_2))
(\chi(A|_{C_1},C_2)+\a_0^{C_2}(Aw_1,Aw_2)\a_{C_1}(u,v))}\\
%=&\frac{(\chi(A|_{C_1},C_2)-1)}{\chi(A|_{C_1},C_2)\frac{1+\a_0^{C_2}(Aw_1,Aw_2)}{\a_0^{C_2}(Aw_1,Aw_2)}
%+(1+\a_0^{C_2}(Aw_1,Aw_2))\a_{C_1}(u,v)}(\a_{C_1}(u,v)-1)\\
\le&\frac{(\chi(A|_{C_1},C_2)-1)}{(\chi(A|_{C_1},C_2)+1)}(\a_{C_1}(u,v)-1).
\end{align*}
The proof is done.
\end{proof}

 Combining with Corollary \ref{C:NormalityAndFocusing}, Lemma \ref{L:ContractionOnCone0} immediately yields the following lemma

\begin{lemma}\label{L:ContractionOnCone1}%L:ContractionOnCone1
 For solid $k$-cones $C_1,C_2$, let $A\in\mc S_{SF}(C_1,C_2)$ and $\tau_{S}(A)=\frac{p(\chi_{S}(A|_{C_1},C_2))-1}{p(\chi_{S}(A|_{C_1},C_2))+1},$ where $p(x)$ is defined in Corollary \ref{C:NormalityAndFocusing}. Then
 $$0\le \tau(A)\le \tau_{V}(A)\le \tau_{S}(A).$$
 \end{lemma}\vskip 1mm
 \begin{proof}
 Since $A\in\mc S_{SF}(C_1,C_2)$, we let  $\d_*=\underline{\di} (A(C_1),X\setminus C_2)>0$. It then follows from a) of Corollary \ref{C:NormalityAndFocusing} that
 $\chi(A|_{C_1},C_2)\le p(\frac{1}{\d_*})$. By the definition of $\chi_{S}(A|_{C_1},C_2)$ in
 \eqref{E:FocusingNumber}, we have $\chi_{S}(A|_{C_1},C_2)=\frac{1}{\d_*};$ and hence, Lemma \ref{L:ContractionOnCone0} implies that  $$\tau_{V}(A)\le \frac{\chi(A|_{C_1},C_2)-1}{\chi(A|_{C_1},C_2)+1}\le \frac{p(\frac{1}{\d_*})-1}{p(\frac{1}{\d_*})+1}=
 \frac{p(\chi_{S}(A|_{C_1},C_2))-1}{p(\chi_{S}(A|_{C_1},C_2))+1}=\tau_{S}(A).$$
 \end{proof}

 %%%%%%%%%       New Section     %%%%%%%%%%%%%%%%%%%%

 \section{Proof of Theorem B}\label{S:Invariant subspace and exponential separation}%S:Invariant subspace and exponential separation

 In this section, we focus on proving the main Theorem \ref{T:ConeImpliesSplit}, Theorem \ref{T:SplitImpliesConeW} and Theorem \ref{T:SplitImpliesCone1d}, by which Theorem B is proved.

\subsection{Proof of Theorem \ref{T:ConeImpliesSplit}}\label{S:Theorem1}%S:Theorem1
In this subsection, we assume  C1)-C4) are satisfied, and for sake of simplenessss, we may assume $(\O,\mc F,\mb P,\t)$ being ergodic, then the dimension of cones becomes constant for which we denote by $\dim E_{0}$.

\vskip 1mm
Firstly, we will investigate the sequence $$\{T^n(\t^{-n}\o)E_0(\t^{-n}\o)\}_{n\ge 0},$$ and  show that such sequence is a Cauchy sequence under the gap metric $d$ (Definition \ref{E:GapDistance}), so that it converges to some $k$-dimensional subspace $\mc E(\o)$ whose properties are summarized in Proposition \ref{P:InvariantSubspace1}.

To this end, for each $\o\in \O$, we define
\begin{equation}\label{D:TauS}%D:TauS
\tau_{S}(\o)=\frac{p(\chi(\o))-1}{p(\chi(\o))+1},
\end{equation}
 where $p(\cdot)$ is defined in Lemma \ref{L:ContractionOnCone1} and $\chi(\cdot)$ is the tempered function defined in Condition C3). Clearly, $\tau_S(\o)$ is a measurable tempered function and satisfies
 \begin{equation}\label{E:TauSBound}%E:TauSBound
 \frac{44}{45}\le \frac{p(1)-1}{p(1)+1}\le \tau_S(\o)<1.
 \end{equation}

\begin{lemma}\label{L:ContractionOnCone2}%L:ContractionOnCone2
There exists a $\t$-invariant set $\tilde \O\subset \O$ with full $\mb P$-measure such that for all $\o\in\tilde \O$, $E,F\in \G_k(C(\o))$,
\begin{align}\begin{split}\label{E:Contraction1}%E:Contraction1
&\limsup_{n\to\infty}\frac1n\log d(T^n(\o)E,T^n(\o)F)\\
\le&\limsup_{n\to\infty}\frac1n\log(\a_{C(\t^n\o)}(T^n(\o)E,T^n(\o)F)-1)\\
\le &\int_\O\log \tau_{S} (\o)d\mb P(\o)\in [\log 44-\log45,0).
\end{split}
\end{align}
\end{lemma}

 \begin{proof}

 The first inequality of (\ref{E:Contraction1}) follows from C4),  III) of Lemma \ref{L:ProjectiveMetric} and the temperedness of $\chi$, it remains to prove the second part.

 By virtue of C3), one has $A(\o)\in \mc S_{SF}(C(\o),C(\t\o))$. Together with \eqref{E:FocusingNumber} and
 Lemma \ref{L:ContractionOnCone2}, we can further deduce from C3) that \begin{equation}\label{E:contraction-eff-comparison}
 \tau(A(\o))\le\tau_{S}(A(\o))\le \tau_{S}(\o).
 \end{equation}
 Note that
\begin{equation}\label{E:Contraction1a}%E:Contraction1a
(\a_{C(\t^n\o)}(T^n(\o)E,T^n(\o)F)-1)\le \tau(T^n(\o))(\a_{C(\o)}(E,F)-1).
\end{equation} Then, by (\ref{E:SubMult}) and \eqref{E:contraction-eff-comparison}, one has
\begin{equation}\label{E:Contraction1b}%E:Contraction1b
\tau(T^n(\o))\le \prod_{i=0}^{n-1}\tau(A(\t^i(\o)))\le \prod_{i=0}^{n-1}\tau_{S}(A(\t^i(\o)))\le\prod_{i=0}^{n-1}\tau_{S}(\t^i(\o)).
\end{equation}
By the Birkhoff's ergodic theorem, there exists a $\t$-invariant set $\tilde \O\subset \O$ of full $\mb P$ measure on which
 \begin{align*}
 &\limsup_{n\to\infty}\frac1n\log(\a_{C(\t^n\o)}(T^n(\o)E,T^n(\o)F)-1)\\
&\le \lim_{n\to\infty}\frac1n\sum_{i=0}^{n-1}\log \tau_{S}(\t^i(\o))=\int_\O\log \tau_{S} (\o)d\mb P(\o),
 \end{align*}
 where $\int_\O\log \tau_{S} (\o)d\mb P(\o)\in[\log 44-\log 45,0)$ by (\ref{E:TauSBound}).
 \end{proof}

By using the exactly same proof of Lemma \ref{L:ContractionOnCone2}, one can also derive a similar result for which we state here without proof.
\begin{lemma}\label{L:ContractionOnCone2'}%L:ContractionOnCone2'
There exists a $\t$-invariant set $\tilde \O\subset \O$ with full $\mb P$-measure such that for all $\o\in\tilde \O$, $E\in \G_k(C(\o))$, and $u\in \tint C(\o)\cap S$,
\begin{align}\begin{split}\label{E:Contraction1'}%E:Contraction1'
&\limsup_{n\to\infty}\frac1n\log \di(T^n(\o)u,T^n(\o)E)\\
\le&\limsup_{n\to\infty}\frac1n\log(\a_{C(\t^n\o)}(T^n(\o)u,T^n(\o)E)-1)\\
\le &\int_\O\log \tau_{S} (\o)d\mb P(\o)\in [\log 44-\log 45,0),
\end{split}
\end{align}
where $\di(u,E)=\inf\left\{|u-v|\big|\ v\in E\right\}$.
\end{lemma}

Lemma \ref{L:ContractionOnCone2} simply says that any two orbits of $T^n(\o)$ on $\G_{k}(C(\o))$ cluster exponentially fast forward in time. However, we still can not say that they converge to a point, which may turns out to be a stationary process.

\begin{lemma}\label{L:ContractionOnCone3}%L:ContractionOnCone3
 There exists a $\t$-invariant set $\tilde \O\subset \O$ of full $\mb P$-measure such that for any $\o\in\tilde \O$, and $E(\o)\in \G_k(C(\o))$
 \begin{align}\begin{split}\label{E:Contraction2}%E:Contraction2
 &\limsup_{n\to\infty}\frac1n\log d(T^{n}(\t^{-n}\o)E(\t^{-n}\o),T^{n+1}(\t^{-(n+1)}\o)E(\t^{-(n+1)}\o))\\
 \le&\limsup_{n\to\infty}\frac1n\log (\a_{C(\o)}(T^{n}(\t^{-n}\o)E(\t^{-n}\o),T^{n+1}(\t^{-(n+1)}\o)E(\t^{-(n+1)}\o))-1)\\
 \le& \int_\O\log \tau_{S} (\o)d\mb P(\o)\in[\log44-\log45,0).
 \end{split}
 \end{align}
 \end{lemma}
\begin{proof}
 The first inequality follows the same argument as in the proof of Lemma \ref{L:ContractionOnCone2}. As for the second inequality, it follows from  C3) and a) of Corollary \ref{C:NormalityAndFocusing} that $\chi\left(A(\t^{n}\o)|_{C(\t^n\o)},C(\t^{n+1}\o)\right)\le p(\chi(\t^n\o))$. Hence, the temperedness of $\chi$ implies that there exists a full $\mb P$-measure set $\tilde \O_{1}$ such that for any $\o\in\tilde \O_{1}$,
 \begin{equation}\label{E:TemperAngle1}%E:TemperAngle1
 \lim_{n\to\pm\infty}\frac1n \log \chi\left(A(\t^{n}\o)|_{C(\t^n\o)},C(\t^{n+1}\o)\right)=0.
 \end{equation}
 Also note that, by Birkhoff's ergodic theorem, there exists another full $\mb P$-measure set $\tilde \O_{2}$ such that for any $\o\in\tilde \O_{2}$,
 $$\lim_{n\to\infty}\frac1n\sum_{i=1}^{n-1}\tau_{S}(\t^{-i}\o)=\int_\O\log \tau_{S} (\o)d\mb P(\o).$$
 Then for any $\o\in\tilde \O:=\tilde \O_{1}\cap\tilde \O_{2}$ and $E(\o)\in \G_{k}(C(\o))$, we have that

 \begin{align*}
 \begin{split}
 &\limsup_{n\to\infty}\frac1n\log\left[\a_{C(\o)}(T^{n}(\t^{-n}\o)E(\t^{-n}\o),
 T^{n+1}(\t^{-(n+1)}\o)E(\t^{-(n+1)}\o))-1\right]\\
 \le& \limsup_{n\to\infty}\frac1n\log\prod_{i=1}^{n-1}\tau(A(\t^{-i}\o))\\
 &\quad\cdot\left[\a_{C(\t^{-n+1}\o)}\left(A(\t^{-n}\o)E(\t^{-n}\o),
 A(\t^{-n}\o)A(\t^{-(n+1)}\o)E(\t^{-(n+1)}\o)\right)-1\right]\\
 \le &\lim_{n\to\infty}\frac1n\log\prod_{i=1}^{n-1}\tau_{S}(\t^{-i}\o)\chi(A(\t^{-n}\o)|_{C(\t^{-n}\o)},C(\t^{-n+1}\o))\\
 \le &\lim_{n\to\infty}\frac1n\log\prod_{i=1}^{n-1}\tau_{S}(\t^{-i}\o)p(\chi(\t^{-n}\o))\\
 =&\int_\O\log \tau_{S} (A(\o))d\mb P(\o)<0.
 \end{split}
\end{align*}
 Note that the $\tilde \O$ does not depends on the choice of $E(\cdot)$, and thus the proof is done.
  \end{proof}

\begin{remark}\label{R:CauchySequence}%R:CauchySequence
Lemma \ref{L:ContractionOnCone3} implies that for any $\o\in\tilde \O$ and $E(\o)\in\G_{k}(C(\o))$, the sequence
$$\{T^{n}(\t^{-n}\o)E(\t^{-n}\o)\}_{n\ge 0}$$ is  a Cauchy sequence in $\G_{k}(A(\t^{-1}\o)(C(\t^{-1}\o)))\subset\G_{k}(C(\o))$. Since
$$\underline\di(A(\t^{-1}\o)(C(\t^{-1}\o)),X\setminus C(\o))>0,$$ such sequence will converge to a $k$-dimensional subspace which belongs to $\tint C(\o)$.
\end{remark}
Now we choose $E(\o)=E_0(\o)$ in Lemma \ref{L:ContractionOnCone3} for all $\o\in\O$, and denote $\mc E(\o)$ the limit of $\{T^{n}(\t^{-n}\o)E_0(\t^{-n}\o)\}_{n\ge 0}$.

\begin{prop}\label{P:InvariantSubspace1}%P:InvariantSubspace1
{\it There exists a $\t$-invariant set $\tilde \O\subset \O$ of full $\mb P$-measure such that the following are satisfied:
\begin{itemize}
\item[i)] $\mc E:\tilde \O\to \G_{k}(C)$ is measurable with respect to the ``past`` $\s$-algebra $\mc F^{-}$ generated by the random variables $(A(\t^{-n}\o))_{n\in\mb N}$;
\item[ii)] For $E(\o)\in\G_{k}(C(\o))$ and $\o\in\tilde \O$,
$$\limsup_{n\to\infty}\frac1n\log d(T^{n}(\t^{-n}\o)E(\t^{-n}\o),\mc E(\o))\le \int_\O\log \tau_{S} (\o)d\mb P(\o)<0;$$
\item[iii)] For all $\o\in\tilde \O$, $A(\o)\mc E(\o)=\mc E(\t\o)$;
\item[iv)] For all $\o\in\tilde \O$ and $E(\o)\in \G_{k}(C(\o))$,
$$\limsup_{n\to\infty}\frac1n\log d(T^{n}(\o)E(\o),\mc E(\t^{n}\o))\le \int_\O\log \tau_{S} (\o)d\mb P(\o)<0.$$
\end{itemize}
}\end{prop}

\begin{proof}
i) The measurability of $\mc E$ is directly from the measurability of $A$ and the measurability of $E_0(\o)$ in condition C2).

ii) For $E(\o)=E_0(\o)$, it follows from Lemma \ref{L:ContractionOnCone3}  that
\begin{align*}
&d(T^{n}(\t^{-n}\o)E_{0}(\t^{-n}\o),\mc E(\o))\\
\le& \sum_{j=n}^{\infty} d(T^{j}(\t^{-j}\o)E_{0}(\t^{-j}\o),T^{j+1}(\t^{-(j+1)}\o)E_{0}(\t^{-(j+1)}\o))<\infty.
\end{align*}
Moreover, Lemma \ref{L:ContractionOnCone3} implies that, for any given $\e>0$ with $\int_\O\log \tau_{S} (\o)d\mb P(\o)+\e<0$, there is a constant $K(\e,\o)$ such that for all $j\in\mb N$,
\begin{align*}
&d(T^{j}(\t^{-j}\o)E_{0}(\t^{-j}\o),T^{j+1}(\t^{-(j+1)}\o)E_{0}(\t^{-(j+1)}\o))\\
\le& K(\e,\o) \exp\left\{j\cdot(\int_\O\log \tau_{S} (\o)d\mb P(\o)+\e)\right\}.
\end{align*}
Thus, by the arbitrariness of $\e$, we obtain that
\begin{equation}\label{G:E0-contrcting}\limsup_{n\to\infty}\frac1n \log d(T^{n}(\t^{-n}\o)E_{0}(\t^{-n}\o),\mc E(\o))\le \int_\O\log \tau_{S} (\o)d\mb P(\o).
\end{equation}
Now for arbitrary $E(\cdot)\in\G_k(C(\cdot))$, the triangle inequality of $d$ implies that
\begin{align*}
&\limsup_{n\to\infty}\frac1n \log d(T^{n}(\t^{-n}\o)E(\t^{-n}\o),\mc E(\o))\\
\le &\limsup_{n\to\infty}\frac1n \log\Big( d(T^{n}(\t^{-n}\o)E(\t^{-n}\o),T^{n}(\t^{-n}\o)E_{0}(\t^{-n}\o))\\
&\quad+d(T^{n}(\t^{-n}\o)E_{0}(\t^{-n}\o),\mc E(\o))\Big)\\
\le&\max
 \Big\{\limsup_{n\to\infty}\frac1n \log d(T^{n}(\t^{-n}\o)E(\t^{-n}\o),T^{n}(\t^{-n}\o)E_{0}(\t^{-n}\o)),\\
&\quad\limsup_{n\to\infty}\frac1n \log d(T^{n}(\t^{-n}\o)E_{0}(\t^{-n}\o),\mc E(\o))\Big\}.
\end{align*}
 By repeating the same arguments in  Lemma \ref{L:ContractionOnCone3}, one can show that
$$\limsup_{n\to\infty}\frac1n \log d(T^{n}(\t^{-n}\o)E(\t^{-n}\o),T^{n}(\t^{-n}\o)E_{0}(\t^{-n}\o))
\le\int_\O\log \tau_{S} (\o)d\mb P(\o).$$ Therefore, together with \eqref{G:E0-contrcting}, we obtain
$$\limsup_{n\to\infty}\frac1n \log d(T^{n}(\t^{-n}\o)E(\t^{-n}\o),\mc E(\o))
\le\int_\O\log \tau_{S} (\o)d\mb P(\o).$$

iii) Note that $A(\o)$ is a bounded operator, so it is also continuous, and hence, it follows that
\begin{align}
A(\o)\mc E(\o)&=A(\o)\left(\lim_{n\to\infty}T^{n}(\t^{-n}\o)E_{0}(\t^{-n}\o)\right) \notag\\
&=\lim_{n\to\infty}T^{n}(\t^{-n+1}\o)A(\t^{-n}\o)E_{0}(\t^{-n}\o) \notag\\
&=\lim_{n\to\infty}T^{n}(\t^{-n+1}\o)E_{0}(\t^{-n+1}\o)=\mc E(\o). \label{ff}
\end{align}
To derive the first equality in \eqref{ff}, letting $E_{n}=A(\t^{-n}\o)E_{0}(\t^{-n}\o)$, it then follows from C3)-C4) and a) of Corollary \ref{C:NormalityAndFocusing} that
\begin{align*}
 &d(T^{n}(\t^{-n}\t\o)E_{n},T^{n}(\t^{-n}\t\o)E_{0}(\t^{-n+1}\o))\\
 \le& 4\chi(\o)\tau(T^{n-1}(\t^{-n+2}\o))\chi(A(\t^{-n+1}\o)|_{C(\t^{-n+1}\o)},C(\t^{-n+2}\o)),
 \end{align*}
where the right hand side converges to zero exponentially fast (c.f. the exactly same argument in the proof of Lemma \ref{L:ContractionOnCone3}). Finally, the last equality in \eqref{ff} directly follows from ii) of this Lemma.

iv) Following the same arguments as in iii), one has $\mc E(\t^{n}\o)=T^{n}(\o)\mc E(\o)$ on $\tilde \O$ and for any $E(\o)\in\G_{k}(C(\o))$
$$ d(T^{n}(\o)E(\o),T^n(\o)E_{0}(\o))\le 4\chi(\t^n\o)\tau(T^{n}(\o))(\a(E(\o),E_{0}(\o))-1).$$
Thus, we have completed the proof of this Proposition.
\end{proof}

\bigskip

%%%%%%%%%%%%%%%%%%%%%%%%%%%%%%%%%%%%%%%%

Secondly, one need to construct the invariant complementary space $\mc E'(\cdot)$ with respect to $\mc E(\cdot)$, which is presented in the following Proposition \ref{P:InvariantSubspace2}. Our approach here is based on the so called graph transform method.

Hereafter, we denote $\G'_{k}(C(\o))$ the collection of all $k$-codimentional subspace $E'$ such that $E'\cap C(\o)=\{0\}$.

\begin{prop}\label{P:InvariantSubspace2}%P:InvariantSubspace2
{\it There exists a $\t$-invariant set $\tilde \O\subset \O$ of full $\mb P$-measure and a function $\mc E'(\cdot):\tilde \O\to \G'_{k}(C(\cdot))$ such that the following are satisfied:
\begin{itemize}
\item[i)] $\mc E'(\cdot):\tilde \O\to \G'_{k}(C(\cdot))$ is measurable in the sense that projections associated with the splitting $X=\mc E(\cdot)\oplus\mc E'(\cdot)$ are strongly measurable;
\item[ii)] For all $\o\in\tilde \O$, $A(\o)\mc E'(\o)\subset\mc E'(\t\o)$.
\end{itemize}
}
\end{prop}

In order to prove Proposition \ref{P:InvariantSubspace2},
we recall from Condition C2) and Proposition \ref{P:InvariantSubspace1} that, for each $\o\in\tilde \O$ with $\tilde \O$ defined in Proposition \ref{P:InvariantSubspace1}, there is a splitting $X=\mc E(\o)\oplus F_0(\o)$.
Denote by $\Pi_1(\o),\Pi_{2}(\o)$ the associated projections onto $\mc E(\o)$ and $F_0(\o)$, respectively. Together with Lemma \ref{L:ProjectionNorm},  C3)-C4) implies that
$$\|\Pi_{1}(\o)\|=\frac1{\underline\di(\mc E(\o),F_0(\o))}\le \chi(\t^{-1}\o),\,\,\,
\|\Pi_{2}(\o)\|=\frac1{\underline\di(F_0(\o),\mc E(\o))}\le \chi(\t^{-1}\o).$$
Note also that $1\ge \underline\di(\mc E(\o),F_0(\o)),\underline\di(F_0(\o),\mc E(\o))$. Then the temperedness of $\chi$ entails that
\begin{equation}\label{E:TemperedProjection}% E:TemperedProjection
\lim_{n\to\pm\infty}\frac1n\log \|\Pi_{i}(\t^{n}\o)\|=0,\, i=1,2,\,\o\in\tilde\O.
\end{equation}

Now we define a one-side cocycle
$$\tilde T^{n}(\o):X\to F_0(\t^n\o);\,\,\tilde T^{n}(\o)\triangleq\Pi_{2}(\t^{n}\o)T^{n}(\o)\Pi_{2}(\o),$$ which is strongly measurable. The following two lemmas are crucial to the proof of Proposition \ref{P:InvariantSubspace2}.
\begin{lemma}\label{L:ExponentialSeparation}%L:ExponentialSeparation
For any $\e>0$, there exists a tempered function $K(\cdot):\O\to [1,\infty)$ such that for any $\o\in\tilde \O$
$$\|\tilde T^{n}(\o)\|\left\|\left(T^{n}(\o)|_{\mc E(\o)}\right)^{-1}\right\|\le K(\o) e^{n(\int_\O\log \tau_{S} (\o)d\mb P(\o)+2\e)},$$
which entails that
for any $\o\in\tilde \O$ and $v\in \mc E(\o)\setminus\{0\}, u\in F_0(\o)\setminus\{0\}$,
$$\liminf_{n\to\infty}\frac1n\log\frac{|T^{n}(\o)v|}{|\tilde T^{n}(\o)u|}\ge -\int_\O\log \tau_{S} (\o)d\mb P(\o)>0.$$
\end{lemma}

\begin{proof}
Given any $\e>0$, we define
$$K_{\e}(\o):=\sup_{n\ge1}\left\{\frac{\prod_{i=0}^{n-1}\tau_{S}(\t^{i}\o)}{e^{n(-\d_1+\e)}}\right\},$$
where we take $\d_1=-\int_\O\log \tau_{S} (\o)d\mb P(\o)>0$. It then follows from Birkhoff's ergodic theorem that $K_{\e}(\cdot)<\infty$ $\mb P$-a.e.. Moreover, a straightforward calculation yields that
$\log K_{\e}(\o)-\log K_{\e}(\t^{-1}\o)\le -\d_1+\e-\log\tau_{S}(\t^{-1}\o)\in L^{1}(\mb P).$
It then follows from Lemma \ref{L:TemperedFunctionA} in the appendix that \begin{equation}\label{E:TemperedFunction2}%E:TemperedFunction2
\lim_{n\to\pm\infty}\frac1n \log K_{\e}(\t^{n}\o)=0.
\end{equation}

For any $E(\o)\in \G_{k}(C(\o))$ and $w\in \tint C(\o)\cap S$, it follows from Condition C3)-C4), b) of Corollary \ref{C:NormalityAndFocusing} and Lemma \ref{L:ContractionOnCone2'} (using inequalities adapted to Lemma  \ref{L:ContractionOnCone2'} which are analogous to (\ref{E:Contraction1a})-(\ref{E:Contraction1b})) that
\begin{equation}\label{E:Contraction1c}
\di\left(\frac{T^{n}(\o)w}{|T^{n}(\o)w|},T^{n}(\o)E(\o)\right)\le 4\chi(\t^n\o)K_{\e}(\o)e^{n(-\d_1+\e)}(\a_{C(\o)}(w,E(\o))-1).
\end{equation}
In particular, by letting $E(\o)=\mc E(\o)$ in \eqref{E:Contraction1c}, one can further deduce from  a) of Corollary \ref{C:NormalityAndFocusing} and II) of Lemma \ref{L:ProjectiveMetric} that,
for any $w\in \tint C(\o)\cap S$,
 \begin{align}\begin{split}\label{E:Contraction1d}%E:Contraction1d
 &\di\left(\frac{T^{n}(\o)w}{|T^{n}(\o)w|},\mc E(\t^n\o)\right)\\
 \le& 848\chi(\t^n\o)K_{\e}(\o)\left(\max\left\{\chi(\t^{-1}\o),(\di(w,X\setminus C(\o)))^{-1}\right\}\right)^{4}e^{n(-\d_1+\e)}\di(w,\mc E(\o)).
 \end{split}
 \end{align}

Now, for any $v\in \mc E(\o)\cap S$ and $u\in F_0(\o)\cap S$, since $\underline\di(\mc E(\o),X\setminus C(\o))\ge (\chi(\t^{-1}\o))^{-1}$, one has
 \begin{equation}\label{E:RangeOf c}%E:RangeOf c
 v+cu\in \tint C(\o)\,\,\,\text{ for any }c\in \left[0,(\chi(\t^{-1}\o))^{-1}\right).
 \end{equation}
In particular, when $c\in\left(0,\frac12(\chi(\t^{-1}\o))^{-1}\right)$, one has $$\di\left(\frac{v+cu}{|v+cu|},X\setminus C(\o)\right)\ge \frac13(\chi(\t^{-1}\o))^{-1}.$$
So, by taking $w=v+cu$ in \eqref{E:Contraction1d} with $c\in\left(0,\frac12(\chi(\t^{-1}\o))^{-1}\right)$, we obtain
 \begin{align}\begin{split}\label{E:Contraction1d-02}%E:Contraction1d
 &\di\left(\frac{T^{n}(\o)(v+cu)}{|T^{n}(\o)(v+cu)|},\mc E(\t^n\o)\right)\\
 \le& 848\chi(\t^n\o)K_{\e}(\o)(3\chi(\t^{-1}\o))^4e^{n(-\d+\e)}\di\left(\frac{v+cu}{|v+cu|},\mc E(\o)\right).
 \end{split}
 \end{align}
On the other hand, it is easy to see from Lemma \ref{L:ProjectionNorm} that
\begin{align*}\begin{split}%E:Contraction1d
 \di\left(\frac{\tilde{T}^{n}(\o)(cu)}{|\tilde{T}^{n}(\o)(cu)|},\mc E(\t^n\o)\right)\ge \frac{1}{\|\Pi_{2}(\t^{n}\o)\|},
 \end{split}
 \end{align*}
and hence,
$$|\tilde{T}^{n}(\o)(cu)|\le \di\left(T^{n}(\o)(v+cu),\mc E(\t^n\o)\right)\cdot\|\Pi_{2}(\t^{n}\o)\|.$$
Together with \eqref{E:Contraction1d-02}, this then implies that
$$
 \frac{|\tilde{T}^{n}(\o)(cu)|}{|T^{n}(\o)(v+cu)|}
 \le 848\chi(\t^n\o)K_{\e}(\o)(3\chi(\t^{-1}\o))^4\|\Pi_{2}(\t^{n}\o)\|
 e^{n(-\d+\e)}\di\left(\frac{v+cu}{|v+cu|},\mc E(\o)\right).
 $$
 Recall that $\di(u,\mc E(\o))\le 1$. Then we have
\[
 \frac{|\tilde{T}^{n}(\o)(u)|}{|T^{n}(\o)(v+cu)|}
 \le 848\chi(\t^n\o)K_{\e}(\o)(3\chi(\t^{-1}\o))^4\|\Pi_{2}(\t^{n}\o)\|
 e^{n(-\d+\e)}(1-|cu|)^{-1},
\]
whenever $c\in\left(0,\frac12(\chi(\t^{-1}\o))^{-1}\right)$.

Letting $c\to 0$ and noticing the arbitrariness of $v,u$,  we have that for all $n\ge 1$,
 \begin{align}\begin{split}\label{E:ExponentialSeparationNI}%E:ExponentialSeparationNI
 &\|\tilde T^{n}(\o)\|\left\|\left(T^{n}(\o)|_{\mc E(\o)}\right)^{-1}\right\|\\
 \le& 848\chi(\t^{n}\o)K_{\e}(\o)(3\chi(\t^{-1}\o))^{4}\|\Pi_{2}(\t^{n}\o)\|e^{n(-\d+\e)}\\
 \le&848K'(\o)K_{\e}(\o)(3\chi(\t^{-1}\o))^{4}e^{n(-\d+2\e)},
 \end{split}\end{align}
 where $K'(\o)$ is a tempered function given by Lemma \ref{L:TemperedFunctionB} in the appendix with $f$ replaced by $\chi(\cdot)\|\Pi_{2}(\cdot)\|$ (see (\ref{E:TemperedProjection})). Actually, we can take
 \begin{equation}\label{E:TemperedFunction1}%E:TemperedFunction1
 K'(\o)= \sup_{n\in\mb Z}\left\{\chi(\t^n\o)\|\Pi_{2}(\t^{n}\o)\|e^{-|n|\e}\right\}.
 \end{equation}
Since $\e$ can be chosen arbitrarily small, together with Condition C3) and (\ref{E:TemperedFunction2}),  we complete the proof by taking
$$K(\o)=848K'(\o)K_{\e}(\o)(3\chi(\t^{-1}\o))^{4}.$$
Thus, we have proved the Lemma.
\end{proof}

%%%%%%%%%%%%%%%%%%%%%%%%%%%%%%%%%%%%%%%%%%%%%%%%%%%%%%%
%The next lemma is another technical lemma which we will use later.
\begin{lemma}\label{L:InverseNorm}%L:InverseNorm
For any $\o\in\tilde \O$,
$$\left\|\left(A(\o)|_{\mc E(\o)}\right)^{-1}\Pi_{1}(\t\o)A(\o)\Pi_{2}(\o)\right\|\le \chi(\t^{-1}\o)\|\Pi_{2}(\o)\|,$$
where the right hand side is tempered.
\end{lemma}
\begin{proof}
For any $u\in F_0(\o)\cap S$,  we will show that
 $$\left|\left(A(\o)|_{\mc E(\o)}\right)^{-1}\Pi_{1}(\t\o)A(\o)u\right|\le \chi(\t^{-1}\o).$$
Without losing any generality, we assume that $\left(A(\o)|_{\mc E(\o)}\right)^{-1}\Pi_{1}(\t\o)A(\o)u\neq 0$, since otherwise it is travail.
Now we take $$v=-\frac{\left(A(\o)|_{\mc E(\o)}\right)^{-1}\Pi_{1}(\t\o)A(\o)u}{\left|\left(A(\o)|_{\mc E(\o)}\right)^{-1}\Pi_{1}(\t\o)A(\o)u\right|}.$$
By (\ref{E:RangeOf c}), we have that
\begin{equation}\label{E:RangeOf c1}%E:RangeOf c1
\text{ for any } c\in[0,(\chi(\t^{-1}\o))^{-1}],\ A(\o)(v+cu)\in\tint C(\t\o).
\end{equation}
To ensure (\ref{E:RangeOf c1}), it is necessary to have that
$c-\frac1{\left|\left(A(\o)|_{\mc E(\o)}\right)^{-1}\Pi_{1}(\t\o)A(\o)u\right|}$ does not change sign for $c\in[0,(\chi(\t^{-1}\o))^{-1}]$. Otherwise, when $c=\frac1{\left|\left(A(\o)|_{\mc E(\o)}\right)^{-1}\Pi_{1}(\t\o)A(\o)u\right|}\in[0,(\chi(\t^{-1}\o))^{-1}]$, then
$$A(\o)(v+cu)=\Pi_2A(\o)(cu)\in F_0(\t\o)-\{0\},$$
which contradicts C1). Thus
$$\left|\left(A(\o)|_{\mc E(\o)}\right)^{-1}\Pi_{1}(\t\o)A(\o)u\right|\le \chi(\t^{-1}\o).$$
By the arbitrariness of $u$ and, the proof is completed.
\end{proof}

%%%%%%%%%%%%%%%%%%%%%%%%%%%%%%%%%%%%%%%%%%%%%%%%%%%%%%%%%May 26 2014%%%%%%%%%%%%%%%%%%%%%%%%%%%%%%

\begin{proof}[Proof of Proposition \ref{P:InvariantSubspace2}.] We will construct $\mc E'(\o)$ by viewing it as the graph of a strongly measurable function $\Psi:\O\to  L(X)$. The proof follows the procedure in \cite{M}, which is also used in \cite{LL}. We request $\Psi$ to satisfy the following properties for any $\o\in\tilde \O$:
\begin{align*}
\quad\quad &\Psi(\o)=\Psi(\o)\Pi_{2}(\o),\\
&\Psi(\o)(F_0(\o))\subset \mc E(\o),\\
&A(\o)\mc G(\Psi(\o))=\mc G(\Psi(\t\o)),
\end{align*}
where $\mc G(\Psi(\o)):=\{v+\Psi(\o)v|\, v\in F_0(\o)\}$ is the graph of $\Psi(\o)$ over $F_0(\o)$.  First, it is not difficult to see the following statements are equivalent:
\begin{align*}
&A(\o)(v+\Psi(\o)v)=v'+\Psi(\t\o)v',\text{ where }v\in F_0(\o),v'\in F_0(\t\o),\\
&\Pi_{1}(\t\o)A(\o)v+A(\o)\Psi(\o)v=\Psi(\t\o)v',\ \tilde T(\o)v=v',\\
&\left(A(\o)|_{\mc E(\o)}\right)^{-1}\Psi(\t\o)\left(\tilde T(\o)|_{F_0(\o)}\right)-\Psi(\o)=\left(A(\o)|_{\mc E(\o)}\right)^{-1}\Pi_{1}(\t\o)\left(A(\o)|_{F_0({\o})}\right).
\end{align*}
Thus, if such graph exists, then it should satisfy the following formula,
\begin{equation}\label{E:GraphFormula}%E:GraphFormula
\Psi(\o)=-\sum_{n=0}^{\infty}\left(T^{n+1}|_{\mc E(\o)}\right)^{-1}\Pi_{1}(\t^{n+1}\o)A(\t^{n}\o)\Pi_{2}(\t^{n}\o)T^{n}(\o)\Pi_{2}(\o).
\end{equation}
The strong measurability of above series follows the same argument as in \cite{LL} once it converges. The (absolute) convergency of above series follows Lemmas \ref{L:ExponentialSeparation}-\ref{L:InverseNorm} and (\ref{E:TemperedProjection}). Furthermore, for a small enough $\e>0$, we apply Lemma \ref{L:TemperedFunctionB} to $f=\chi(\t^{-1}\cdot)\|\Pi_{2}(\cdot)\|$ (which is the upper bound of $\|(A(\cdot)|_{\mc E(\cdot)})^{-1}\Pi_{1}(\t\cdot)A(\cdot)\Pi_{2}(\cdot)\|$ by Lemma \ref{L:InverseNorm}) and then obtain that
\begin{equation}\label{E:GraphNorm}%E:GraphNorm
\|\Psi(\o)\|\le K(\o)R(\o)\frac1{1-e^{\int_\O\log \tau_{S} (\o)d\mb P(\o)+3\e}},
\end{equation}
where $R(\omega)$ is from Lemma \ref{L:TemperedFunctionB} (with $f=\chi(\t^{-1}\cdot)\|\Pi_{2}(\cdot)\|$). Therefore, the right hand side is tempered.

\medskip
Now let $$\pi'(\o)=\Pi_{2}(\o)+\Psi(\o),\ \pi(\o)=I-\pi'(\o),\text{ and }\, \mc E'(\o)=\mc G(\Psi(\o)).$$
It is easy to see that i) and ii) of Proposition \ref{P:InvariantSubspace2} follows the procedure of constructing $\Psi$, once one can show that $\mc E'(\o)\cap C(\o)=\{0\}$.

To show $\mc E'(\o)\cap C(\o)=\{0\}$, first, we show that $\pi(\o),\pi'(\o)$ are projections associated to the splitting $X=\mc E(\o)\oplus \mc E'(\o)$. Note that, by definition,
$$\mc E'(\o)=\pi'(\o)X.$$
And on the other hand, for any $u\in \mc E'(\o)$, by definition, there exists $u'\in F_{0}(\o)$ such that
$$u=u'+\Psi(\o)u'.$$
Then
\begin{align*}
\pi'(\o)u&=(\Pi_{2}(\o)+\Psi(\o))(u'+\Psi(\o)u')\\
&=(\Pi_{2}(\o)+\Psi(\o))u'+(\Pi_{2}(\o)+\Psi(\o))\Psi(\o)u'\\
&=u'+\Psi(\o)u'+0+0=u.
\end{align*}
So, $\pi'(\o)$ is indeed a projection from $X$ to $\mc E'(\o)$; and hence,  $\pi(\o)$ is a projection from $X$ to $\mc E(\o)$.

By the temperedness of $\|\Pi_{2}(\cdot)\|$ and $\|\Psi(\cdot)\|$, we have that $\|\pi(\cdot)\|,\|\pi'(\cdot)\|$ are tempered, which, together with Lemma \ref{L:ProjectionNorm}, imply that
\begin{equation}\label{E:TemperedAngleInvariant}%E:TemperedAngleInvariant
\lim_{n\to\infty}\frac1n\log\underline\di(\mc E'(\t^{n}\o),\mc E(\t^{n}\o))=0,
\end{equation}
which, together with Lemma \ref{L:ContractionOnCone2'}, imply that $\mc E'(\cdot)\cap C(\cdot)=\{0\}$. Thus, we have completed the proof of Proposition \ref{P:InvariantSubspace2}.
\end{proof}

\medskip
Now, we are ready to prove Theorem \ref{T:ConeImpliesSplit}.

\begin{proof}[Proof of Theorem \ref{T:ConeImpliesSplit}.]%Proof of Theorem \ref{T:ConeImpliesSplit}

D1) follows Propositions \ref{P:InvariantSubspace1} and \ref{P:InvariantSubspace2} straightforwardly, if one take $\tilde \O$ to be the intersection of $\tilde \O$'s derived in Propositions \ref{P:InvariantSubspace1} and \ref{P:InvariantSubspace2}, and $E(\o)=\mc E(\o)$, $F(\o)=\mc E'(\o)$ for all $\o\in\tilde \O$.  \\

D2) Noting that
$$2\ge\underline\di (E(\o),F(\o))= \underline\di(A(\t^{-1}\o)E(\t^{-1}\o),F(\o))\ge \frac1{\chi(\t^{-1}\o)},$$
by which and Condition C3), the proof of D2) has been completed.\\
Furthermore, as a consequence of Lemma \ref{L:ProjectionNorm}, one also has that
\begin{equation}\label{E:TemperAngle'}%E:TemperAngle'
\lim_{n\to\pm\infty}\frac1n\log\underline\di(F(\t^{n}(\o)),E(\t^{n}(\o)))=0.
\end{equation}

D3) follows from Lemma \ref{L:ExponentialSeparation} directly if we take $\mc E'(\o)$ to replace $F_0(\o)$, and one can take $\d=-\int_\O\log \tau_{S} (\o)d\mb P(\o)-2\e>0$ for any sufficient small $\e$.
Thus, We have completed the proof for ergodic case.

For the non-ergodic case, all the constants we derived here become $\t$-invariant functions. In particular,
$$\overline\tau_{S}(\o):=\lim_{n\to\pm\infty}\frac1n\sum_{i=0}^{n-1}\tau_{S}(\t^i\o)\text{ exists and }\overline\tau_{S}(\t\o)=\overline\tau_{S}(\o)\in[\log44-\log45,0)\ \mb P-a.e. ,$$
and
$$\d(\o):=-\overline\tau_{S}(\o)-2\e(\o)>0,\text{ where }0<\e(\o)<<-\overline\tau_{S}(\o) \text{ is } \t-\text{invariant}.$$  We have completed the proof.
\end{proof}

%%%%%%%%%%%%%%%%%%%%%  Subsection   %%%%%%%%%%%%%%%%%%%%%%%%%%%%%%%%%%%%%%
\subsection{Proof of Theorem \ref{T:SplitImpliesConeW} and \ref{T:SplitImpliesCone1d}}\label{S:Theorem3&5}%S:Theorem3&5

In this subsection, we will prove Theorems \ref{T:SplitImpliesConeW} and \ref{T:SplitImpliesCone1d}. Actually, they follows the same way of constructing (eventually) measurable cone family from  a measurably dominated splitting $X=E(\o)\oplus F(\o)$ satisfying D1)-D3).

Let $K(\o)$ and $\d(\o)$ be defined in D3). Then, for a given $\t$-invariant function $\e:\O\to (0,\d)$ with $\e(\cdot)\ll \d(\cdot)$, by Lemma \ref{L:TemperedFunctionB}, there is a function $K':\O\to [1,\infty)$ such that
\begin{equation}\label{E:TemperControl}%E:TemperControl
 K(\o)\le K'(\o)\text{ and } e^{-\e(\o)}K'(\o)\le K'(\t\o)\le e^{\e(\o)}K'(\o).
\end{equation}

For any $\o\in\tilde \O$ and $v\in X\setminus \{0\}$, we define an index
\begin{equation}\label{E:IndexZeta}%E:IndexZeta
\zeta_\o(v)=\begin{cases}
\infty,&\text{ for }v\in F(\o)\\
0,&\text{ for }v\in E(\o)\\
\sum_{n=0}^\infty\frac{|T^n(\o)v^F(\o)|}{|T^n(\o)v^E(\o)|}e^{\frac12n\d(\o)}, &\text{ otherwise}
\end{cases},
\end{equation}
where we denote $v^E(\o)=\pi_E(\o)v$ and $v^F(\o)=\pi_F(\o)v$.

By virtue of (\ref{E:UniformSeparation}) in D3) and the property of $K'$, it is easy to see that
\begin{equation}\label{E:Zeta1}%E:Zeta1
\frac{|v^F(\o)|}{|v^E(\o)|}\le \zeta_\o(v)\le \left(1+\frac{K'(\o)e^{-\frac12\d(\o)+\e(\o)}}{1-e^{-\frac12\d(\o)+\e(\o)}}\right)\frac{|v^F(\o)|}{|v^E(\o)|}, \text{ for }v\notin F(\o).
\end{equation}

%The next lemma gives the contracting property of $\zeta$ under interations.
\begin{lemma}\label{L:IndexPro}%L:IndexPro
For any $v\in X$ and $k\ge 1$,
$$\zeta_{\t^k\o}(T^k(\o)v)\le e^{-\frac12k\d(\o)}\zeta_{\o}(v).$$
\end{lemma}
\begin{proof}
It is trivial when $v\in E(\o)$ or $F(\o)$. For $v\notin E(\o), F(\o)$, a direct calculation yields that
\begin{align*}
\zeta_{\t\o}(T(\o)v)=&\sum_{n=0}^\infty\frac{|T^{n+1}(\o)v^F(\o)|}{|T^{n+1}(\o)v^E(\o)|}e^{\frac12n\d(\o)}\\
=&e^{-\frac12\d(\o)}\left(\sum_{n=1}^\infty\frac{|T^n(\o)v^F(\o)|}{|T^n(\o)v^E(\o)|}e^{\frac12n\d(\o)}\right)\\
\le&e^{-\frac12\d(\o)}\left(\sum_{n=0}^\infty\frac{|T^n(\o)v^F(\o)|}{|T^n(\o)v^E(\o)|}e^{\frac12n\d(\o)}\right)\\
\le&e^{-\frac12\d(\o)}\zeta_{\o}(v).
\end{align*}
Thus, the inequality holds for $k=1$. By an induction on $k$, one can obtain the inequality for $k>1$.
\end{proof}
Now for each $\o\in \tilde{\O}$, we define
\begin{equation}\label{E:WeakCone}%E:WeakCone
C(\o)=\{v\in X:\zeta_\o(v)\le 1\}.
\end{equation}
%Here $1$ can be replaced by any positive tempered function satisfying (\ref{E:TemperControl}).
By Lemma \ref{L:IndexPro} and the definition of $\zeta$ in (\ref{E:IndexZeta}), C1) and C2) are satisfied, because we have that
$$E(\o)\subset C(\o),\, F(\o)\cap C(\o)=\{0\}, \text{ and }\, T(\o)C(\o)\subset C(\t\o).$$
\vskip 3mm

\begin{proof}[Proof of Theorem \ref{T:SplitImpliesConeW}]
In order to prove Theorem \ref{T:SplitImpliesConeW}, it suffices to show C3') is satisfied.

Note that for any $n\ge 1$, $w\in T^n(\o)C(\o)\cap S$ and $u\in X\setminus C(\t^n\o)$, it is clear that
\begin{equation}\label{E:Zeta3}%E:Zeta3
|w-u|\ge\frac1{\|\pi_E(\t^n\o)\|+1}\max\left\{|\pi_E(\t^n\o)(w-u)|,|\pi_F(\t^n\o)(w-u)|\right\}.
\end{equation}
Moreover, it follows from  Lemma \ref{L:IndexPro} and (\ref{E:Zeta1}) that
\begin{align}\begin{split}\label{E:Zeta2}%E:Zeta2
|\pi_E(\t^n\o)w|&\ge e^{\frac12n\d(\o)}|\pi_F(\t^n\o)w|\\
|\pi_E(\t^n\o)u|&\le \left(1+\frac{K'(\t^n\o)e^{-\frac12\d(\o)+\e(\o)}}{1-e^{-\frac12\d(\o)+\e(\o)}}\right)|\pi_F(\t^n\o)u|.
\end{split}
\end{align}
By virtue of (\ref{E:TemperControl}), there exists some $l(\o)\in \mb N$ such that, for any $n\ge l(\o)$,
\begin{equation}\label{E:K-Prim-growth}
e^{\frac12n\d(\o)}>\max\left\{7, 8\left(1+\frac{K'(\t^n\o)e^{-\frac12\d(\o)+\e(\o)}}{1-e^{-\frac12\d(\o)+\e(\o)}}\right)\right\}.
\end{equation}
Let $N(\o)$  be the infimum of such $l(\o)$, which is obviously measurable. For any $n\ge N(\o)$, by the first equality of (\ref{E:Zeta2}), we have $\frac76>|\pi_E(\t^n\o)w|>\frac78$.

Now we consider two cases: a) $|\pi_E(\t^n\o)(w-u)|\ge\frac14$;
and b) $|\pi_E(\t^n\o)(w-u)|<\frac14$. \\
 For case a), (\ref{E:Zeta3}) directly entails that
 $$|w-u|\ge\frac1{4(\|\pi_E(\t^n\o)\|+1)}.$$
 If case b) holds, then $|\pi_E(\t^n\o)u|>|\pi_E(\t^n\o)w|-\frac14$. Therefore,
 it follows from \eqref{E:Zeta2}, \eqref{E:K-Prim-growth} and $|\pi_E(\t^n\o)w|>\frac78$ that
 \begin{align*}
|\pi_F(\t^n\o)(w-u)|&\ge |\pi_F(\t^n\o)u|-|\pi_F(\t^n\o)w|\\
&\ge \left(1+\frac{K'(\t^k\o)e^{-\frac12\d(\o)+\e(\o)}}{1-e^{-\frac12\d(\o)+\e(\o)}}\right)^{-1}
 \left(|\pi_E(\t^n\o)u|-\frac18|\pi_E(\t^n\o)w|\right)\\
 &\ge \frac12\left(1+\frac{K'(\t^k\o)e^{-\frac12\d(\o)+\e(\o)}}{1-e^{-\frac12\d(\o)+\e(\o)}}\right)^{-1}.
\end{align*}
Hence, by (\ref{E:Zeta3}) again, one has
  $$|w-u|\ge\frac1{2(\|\pi_E(\t^n\o)\|+1)}\left(1+\frac{K'(\t^k\o)e^{-\frac12\d(\o)+\e(\o)}}{1-e^{-\frac12\d(\o)+\e(\o)}}\right)^{-1}.$$
So, we have that for any $n\ge N(\o)$,
$$|w-u|\ge\frac1{\|\pi_E(\t^n\o)\|+1}\min\left\{\frac14,
\frac12\left(1+\frac{K'(\t^n\o)e^{-\frac12\d(\o)+\e(\o)}}{1-e^{-\frac12\d(\o)+\e(\o)}}\right)^{-1}\right\},$$
which implies C3') by the temperedness of $\|\pi_{E}\|$ and $K'$ and arbitrariness of $w,u$. Thus, we have completed the proof of Theorem \ref{T:SplitImpliesConeW}.
\end{proof}
\vskip 2mm

\begin{proof}[Proof of Theorem \ref{T:SplitImpliesCone1d}]
To prove Theorem \ref{T:SplitImpliesCone1d}, we need to prove that C3) is satisfied.

We again use the cones $C(\o)$ defined in (\ref{E:WeakCone}). We will show that $\zeta_{\o}$ is Lipchitz continuous on a neighborhood of $C(\o)\cap S$ with the Lipchitz contant function (depending on $\o$) being tempered. In fact, for any $w\in C(\o)\cap S$ (hence, $|w^{E}(\o)|\ge\frac 12$) and any $u\in S$ with $|u-w|<\frac1{4\|\pi_{E}(\o)\|}$ (hence, $|u^{E}(\o)|\ge\frac 14$)),
\begin{align*}
&|\zeta_{\o}(w)-\zeta_{\o}(u)|\\
\le&\sum_{n=0}^{\infty}\left|\frac{|T^{n}(\o)w^{F}(\o)|}{|T^{n}(\o)w^{E}(\o)|}-\frac{|T^{n}(\o)u^{F}(\o)|}{|T^{n}(\o)u^{E}(\o)|}\right|\\
\le&\sum_{n=0}^{\infty}\left(\left|\frac{|T^{n}(\o)w^{F}(\o)-T^{n}(\o)u^{F}(\o)|}
{|T^{n}(\o)w^{E}(\o)|}\right|+\frac{|T^{n}(\o)u^{F}(\o)|}
{|T^{n}(\o)w^{E}(\o)|}\left|1-\frac{|T^{n}(\o)w^{E}(\o)|}{|T^{n}(\o)u^{E}(\o)|}\right|\right).
\end{align*}
Since $\dim E(\o)=1$, one has
$$\left|1-\frac{|T^{n}(\o)w^{E}(\o)|}{|T^{n}(\o)u^{E}(\o)|}\right|
\le \frac{|T^{n}(\o)(u^{E}(\o)-w^{E}(\o))|}{|T^{n}(\o)u^{E}(\o)|}= \frac{|(u^{E}(\o)-w^{E}(\o))|}{|u^{E}(\o)|}.$$
Consequently, together with D3), we obtain
\begin{align}
&|\zeta_{\o}(w)-\zeta_{\o}(u)| \notag\\
\le &\sum_{n=0}^{\infty}\left(2K(\o)e^{-n\d(\o)}|w^{F}(\o)-u^{F}(\o)|+2K(\o)\|\pi_{F}(\o)\|e^{-n\d(\o)}\left|1-\frac{w^{E}(\o)}{u^{E}(\o)}\right|\right)\notag\\
\le&\frac{8K(\o)\|\pi_{F}(\o)\|}{1-e^{-\d(\o)}}\left(|w^{F}(\o)-u^{F}(\o)|+|w^{E}(\o)-u^{E}(\o)|\right)\notag\\
\le& \frac{16K(\o)(\|\pi_E(\o)\|+1)^{2}}{1-e^{-\d(\o)}}|w-u|,\quad \,\,\,
\text{ whenever }|u-w|<\frac1{4\|\pi_{E}(\o)\|}. \label{A:zeta-Lipsch}
\end{align}

Therefore, by \eqref{A:zeta-Lipsch} and Lemma \ref{L:IndexPro}, we obtain that, for any $w\in T(\o)C(\o)\cap S$ and $u\in (X\setminus C(\o))\cap S$,
$$|w-u|\ge\min\left\{\frac1{4\|\pi_{E}(\t\o)\|},\ \frac{(1-e^{-\frac12\d(\o)})(1-e^{-\d(\o)})}{16K(\t\o)(\|\pi_E(\t\o)\|+1)^{2}}\right\}.$$ Thus,
we have completed the proof of Theorem \ref{T:SplitImpliesCone1d} by letting
$$\chi(\o)=\max\left\{8\|\pi_{E}(\t\o)\|,\ \frac{32K(\t\o)(\|\pi_E(\t\o)\|+1)^{2}}{(1-e^{-\frac12\d(\o)})(1-e^{-\d(\o)})}\right\}$$
which is obviously tempered.
\end{proof}

 \section{Proof of Theorem \ref{T:ConeImpliesSplitW}}\label{S:TheoremECone}%S:TheoremECone\label{S:Invariant subspace and exponential separation-2}%S:Invariant subspace and exponential separation
In this section, we will prove Theorem \ref{T:ConeImpliesSplitW} by applying Theorem \ref{T:ConeImpliesSplit} to a modified system.

\begin{proof}[Proof of Theorem \ref{T:ConeImpliesSplitW}]
For each $n>m\ge 1$, let
$$\O_{m,n}=\left\{\o\in\O\big|\ m\le N(\o)\le n\right\}.$$
By virtue of C3'), $\O_{m,n}$'s are measurable sets and $\lim_{n\to\infty}\mb P(\O_{1,n})=1$. Thus for any given $\e\in(0,1)$, there exists $n>m\ge 1$ such that $\mb P(\O_{m,n})>1-\e$. By Poincar\'{e} Recurrence Theorem, for $\mb P$-a.e. $\o\in\O_{m,n}$, the orbit staring at $\o$ meets $\O_{m,n}$ infinite times. For such $\o$, we define the first-return function
$$\tau_{m,n}(\o)=\min\{j|\ \t^{j}\o\in \O_{m,n}, \ j\ge 1\}.$$
It is clearly that $\tau_{m,n}$ is measurable.

The pair $(\O_{m,n},\tau_{m,n})$ induces a new system $(\O_{m,n},\mc F_{m,n}, \mb P_{m,n},\t_{m,n},T_{m,n},X)$  as follows:
\begin{align*}
&\mc F_{m,n}=\{U\cap \O_{m,n}|\ U\in \mc F\};\ \mb P_{m,n}=\frac{\mb P}{\mb P(\O_{m,n})};\ \t_{m,n}(\cdot)=\t^{\tau_{m,n}(\cdot)}(\cdot);\\
&A_{m,n}(\cdot)=T^{\tau_{m,n}(\cdot)}(\cdot), \,\, T^j_{m,n}(\cdot)=A_{m,n}(\t_{m,n}^{j-1}(\cdot))\circ\cdots\circ A_{m,n}(\cdot),
\end{align*}
which is the restriction of  $(\O,\mc F, \mb P,\t,T,X)$ on $\O_{m,n}$.

 It is clear that $(\O_{m,n},\mc F_{m,n}, \mb P_{m,n},\t_{m,n})$ is measure preserving system, and is ergodic provided that  $(\O,\mc F, \mb P,\t)$ is ergodic. We will apply Theorem \ref{T:ConeImpliesSplit}  to the system $(\O_{m,n},\mc F_{m,n}, \mb P_{m,n},\t_{m,n}^{n},T^{n}_{m,n},X)$. However, before doing  that, we need to check  that $(\O_{m,n},\mc F_{m,n}, \mb P_{m,n},\t_{m,n}^{n},T^{n}_{m,n},X)$ satisfies C1)-C3). It is obvious that C1) and C2) are automatically satisfied. As for C3), it follows from C3') and the definition of $\O_{m,n}$ that, for any $\mb P_{m,n}$-a.e. $\o\in \O_{m,n}$,
\begin{equation*}\label{E:N-bdd-by-n}
\udi(T^{n}_{m,n}(\o)C(\o),X\setminus C(\t_{m,n}^{n}\o))\ge\frac1{\chi(\t_{m,n}^{n}\o)}.
\end{equation*}
Denote that $\chi'(\o)=\chi(\t_{m,n}^{n}\o)\ge 1$ for $\o\in \O_{m,n}$.
It remains to show that
\begin{equation}\label{E:N-bdd-by-n-tempered}
\lim_{j\to\pm\infty}\frac{1}{|j|}\log \chi'(\t_{m,n}^{jn}\o)=0.
\end{equation}
To this end, we first note that,  by the Kac's Recurrence Theorem,
$$\int_{\O_{m,n}}\tau_{m,n} d\mb P_{m,n}\le\frac1{\mb P(\O_{m,n})}$$ and hence,
the Birkhoff Ergodic Theorem implies that there exists a $\t_{m,n}$-invariant function $\overline\tau_{m,n}:\O_{m,n}\to [1,\infty)$ such that
\begin{equation}\label{E:ReturnFrequency}%E:ReturnFrequency
\lim_{j\to\pm\infty}\frac1{|j|}\sum_{i=0}^{j-1}\tau_{m,n}(\t_{m,n}^{i}\o)=\overline\tau_{m,n}(\o), \ \mb P-a.e.\,\, \o
\end{equation}
and
$$\int_{\O_{m,n}}\overline\tau_{m,n} d\mb P_{m,n}=\int_{\O_{m,n}}\tau_{m,n} d\mb P_{m,n}\le\frac1{\mb P(\O_{m,n})}.$$
As a consequence, for $\mb P_{m,n}$-a.e. $\o\in\O_{m,n}$, we have
\begin{align*}
0&\le\lim_{j\to\pm\infty}\frac1{|j|}\log \chi'(\t_{m,n}^{jn}\o)\\
&=\lim_{j\to\pm\infty}\frac1{|j|}\log\chi(\t_{m,n}^{(j+1)n}\o)\\
&=\lim_{j\to\pm\infty}\frac{\sum_{i=0}^{{\rm sgn}(j)(|j|+1)n}\tau_{m,n}(\t_{m,n}^{i}\o)}{|j|}\cdot\frac1{\sum_{i=0}^{{\rm sgn}(j)(|j|+1)n}\tau_{m,n}(\t_{m,n}^{i}\o)}\log\chi(\t_{m,n}^{(j+1)n}\o)\\
&=\lim_{j\to\pm\infty}\frac{\sum_{i=0}^{{\rm sgn}(j)(|j|+1)n}\tau_{m,n}(\t_{m,n}^{i}\o)}{|j|}\\
&\quad\cdot\frac1{\sum_{i=0}^{{\rm sgn}(j)(|j|+1)n}\tau_{m,n}(\t_{m,n}^{i}\o)}
\log\chi\left(\t^{\sum_{i=0}^{(j+1)n}\tau_{m,n}(\t_{m,n}^i\o)}\o\right)\\
&\le n\overline\tau_{m,n}(\o)\cdot\lim_{j\to\pm\infty}\frac1{|j|}\log\chi(\t^{j}\o)=0.
\end{align*}
Thus, we have proved \eqref{E:N-bdd-by-n-tempered}, which entails that Theorem \ref{T:ConeImpliesSplit} is applicable to system $(\O_{m,n},\mc F_{m,n}, \mb P_{m,n},\t_{m,n}^{n},T^{n}_{m,n},X)$. So, there exists some $\t_{m,n}^{n}$-invariant set $\tilde \O_{m,n}\subset\O$ with $\mb P(\tilde \O_{m,n})=\mb P(\O_{m,n})$, on which there exist a measurable splitting $X=E_{m,n}(\cdot)\oplus F_{m,n}(\cdot)$, two $\t_{m,n}^{n}$-invariant measurable functions $k_{m,n}:\O_{m,n}\to\mb N$ and $\d_{m,n}:\O_{m,n}\to(0,\infty)$, as well as a measurable function $K_{m,n}:\O_{m,n}\to[1,\infty)$ such that D1)-D3) are satisfied for system $(\O_{m,n},\mc F_{m,n}, \mb P_{m,n},\t_{m,n}^{n},T^{n}_{m,n},X)$ by which we denote by the notations D1)$_{m,n}$--D3)$_{m,n}$.

Now choose a sequence or a finite set of positive integers $\{m_{p}\}_{ 0\le p< q}$ with $q\in\mb N\cup \{\infty\}$ such that $1=m_{0}<m_{1}<\cdots<m_{q}$, and for each $p\in[0,q)$, $\mb P(\O_{m_{p},m_{p+1}})>0$ and $\mb P\left(\cup_{p=0}^{q}\O_{m_{p},m_{p+1}}\right)=1$. Let
$$\tilde \O=\bigcap_{i=-\infty}^{+\infty}\t^{i}\left(\bigcup_{p=0}^{q}\tilde\O_{m_{p},m_{p+1}}\right),$$
which is a $\t$-invariant $\mb P$-full measure set, and on which the splitting $X=E(\cdot)\oplus F(\cdot)$ and $\d,K$ are well defined by taking $E(\o)=E_{m_{p},m_{p+1}}(\o)$, $F(\o)=F_{m_{p},m_{p+1}}(\o)$, $\d(\o)=\d_{m_{p},m_{p+1}}(\o)$ and $K(\o)=K_{m_{p},m_{p+1}}(\o)$ when $\o\in\tilde\O_{m_{p},m_{p+1}}$ .

We {\it claim that such splitting $X=E(\cdot)\oplus F(\cdot)$ independent on the choice of $\{m_{p}\}_{0\le p\le q}$ and is $\t$-invariant.}  The proof of the claim will be postponed to the following two lemmas (Lemmas \ref{L:TInvariantSplit1} and \ref{L:TInvariantSplit2}). If the claim is confirmed, then the invariant splitting $X=E(\cdot)\oplus F(\cdot)$ is what we needed. It also satisfies  D1),D2) and D3'). In fact, for such splitting,
 D1)-D2) directly from D1)$_{m_p,m_{p+1}}$--D2)$_{m_p,m_{p+1}}$; while
 D3) follows from D3)$_{m_p,m_{p+1}}$ and (\ref{E:ReturnFrequency}) by taking
$n_{j}=\sum_{i=0}^{j-1}\tau_{m_{p},m_{p+1}}(\t^{i}_{m_{p},m_{p+1}}\o)$ for $\o\in\tilde \O_{m_{p},m_{p+1}}$ and $j\ge1$. Moreover,
$$\lim_{j\to\infty}\frac{j}{n_{j}(\o)}=\frac1{\overline\tau_{m_{p},m_{p+1}}(\o)}>0,\ \text{ for all }\o\in \tilde \O_{m_{p},m_{p+1}}(\o).$$
Thus, we have completed the proof of Theorem \ref{T:ConeImpliesSplitW}.
\end{proof}
\vskip 2mm

\begin{lemma}\label{L:TInvariantSplit1}%L:TInvariantSplit1
Given any $0\le m_{1}\le m_{2}\le n_{2}\le n_{1}<\infty$. If $\mb P(\O_{m_{2},n_{2}})>0$, then
$$E_{m_{2},n_{2}}(\o)=E_{m_{1},n_{1}}(\o),\,\, F_{m_{2},n_{2}}(\o)=F_{m_{1},n_{1}}(\o)\,\text{ for } \mb P-a.e.\ \o\in\tilde \O_{m_{2},n_{2}}.$$
\end{lemma}
\begin{proof}
Note that for any $ j_{2}\ge 1$, there exists $j_{1}\ge 1$ such that for $\mb P$-a.e. $\o\in\tilde \O_{m_{2},n_{2}}$
$$s_{j_{1}}(\o):=\sum_{i=0}^{j_{1}}\tau_{m_{1},n_{1}}(\t_{m_{1},n_{1}}^{-in_{1}}\o)\ge s_{j_{2}}(\o):= \sum_{i=0}^{j_{2}}\tau_{m_{2},n_{2}}(\t_{m_{2},n_{2}}^{-in_{2}}\o).$$
By C1), we have that $$T^{s_{j_{1}}(\o)-s_{j_{2}}(\o)}(\t^{-s_{j_{1}}(\o)}\o)E_{m_{1},n_{1}}(\t^{-s_{j_{1}}(\o)}\o)\subset C(\t^{-s_{j_{2}}(\o)}\o),$$
which implies that
 $$E_{m_{1},n_{1}}(\o)\subset T^{s_{j_{2}}(\o)}C(\t^{-s_{j_{2}}(\o)}\o)\subset C(\o).$$
Applying D3)$_{m_2,n_2}$ to the system $(\O_{m_{2},n_{2}},\mc F_{m_{2},n_{2}}, \mb P_{m_{2},n_{2}},\t_{m_{2},n_{2}}^{n_{2}},T^{n_{2}}_{m_{2},n_{2}},X)$, we have that  $E_{m_{1},n_{1}}(\o)$ converges to $E_{m_{2},n_{2}}(\o)$ (by Proposition \ref{P:InvariantSubspace1} for system $T^{n_{2}}_{m_{2},n_{2}}$) exponentially fast as $j_{2}\to\infty$. Consequently, they are equal to each other.

For the second equality, suppose there is some $\o\in \tilde\O_{m_2,n_2}$ and a vector $v\in F_{m_{1},n_{1}}(\o)\setminus\{0\}$ but $v\notin F_{m_{2},n_{2}}(\o)$.  Then we write $v=w+u$ with $w\in E_{m_{2},n_{2}}(\o)\setminus\{0\}$ and $u\in F_{m_{2},n_{2}}(\o)$. By property D3)$_{m_{2},n_{2}}$, we have that $$T^{in_2}_{m_{2},n_{2}}(\o)v\in C(\t^{in_2}_{m_{2},n_{2}}\o)$$ for some large $i$. So, by C1), one has $T^l(T^{in_2}_{m_{2},n_{2}}(\o)v)\subset C(\t^l\cdot\t^{in_2}_{m_{2},n_{2}}\o)$ for all $l\ge 0$, which implies that  $T^{l}(\o)v\in C(\t^{l}\o)$ for all $l$ sufficient large. In particular, we have
$$T^{in_{1}}_{m_{1},n_{1}}(\o)v\in C(\t_{m_{1},n_{1}}^{in_{1}}\o),$$ for all $i$ sufficient large.
 So, by applying Lemma \ref{L:ContractionOnCone2'} to $(\O_{m_{1},n_{1}},\mc F_{m_{1},n_{1}}, \mb P_{m_{1},n_{1}},\t_{m_{1},n_{1}}^{n_{1}},T^{n_{1}}_{m_{1},n_{1}},X)$, we obtain a contradiction to
 D2)$_{m_1,n_1}$. Thus, the proof is done.
\end{proof}

\begin{lemma}\label{L:TInvariantSplit2}%L:TInvariantSplit2
For all $\o\in\tilde \O$,
$$A(\o)E(\o)=E(\t\o),\ A(\o)F(\o)\subset F(\t\o).$$
\end{lemma}
\begin{proof}
The proof of Lemma \ref{L:TInvariantSplit2} is employing the exact same idea used in the proof of Lemma \ref{L:TInvariantSplit1}. Here we just use the fact that $N(\o)\le N(\t\o)+1$ which follows from C1) and the definition of $N(\cdot)$ in C3'). Thus for any $\o\in \tilde \O_{m,n}$, $\t\o\in \tilde \O_{m-1,\max\{n,N(\t\o)\}}$. For the first equality, we can use the exact same proof as proving Lemma \ref{L:TInvariantSplit1} by taking $m_{1}=m-1,m_{2}=m, ,n_{2}=n, n_{1}=\max\{n,N(\t\o)\}$, and replacing $s_{j_{2}}(\o)$ by $s_{j_{2}}(\t\o)-1$. \\
For the second equality, the proof can be done by doing a similar modification on the proof of Lemma \ref{L:TInvariantSplit1} too. We leave the detailed proof to readers.
\end{proof}

\bigskip

%%%%%%%%%%%%%%%%%%%% Subsection %%%%%%%%%%%%%%%%%%%%%%%%%%%%%%%%%
 \section{Proof of Theorems A and C}\label{S:Invariant subspace and exponential separation-2}%S:Invariant subspace and exponential separation

 In this section, we will first prove Theorem A in Subsection \ref{S:TheoremA}. Then we will prove Theorem \ref{T:SplitImpliesConeL1} and Theorem \ref{T:Cone&SplitMET} in Subsection \ref{S:Theorem4}. Clearly, Theorem C is just a part of Theorem \ref{T:Cone&SplitMET}.

%%%%%%%%%%%%%%%%%%% Subsection %%%%%%%%%%%%%%%%%%%%%%%%%%%%%%%%%%%%%%%%
\subsection{Proof of Theorem A}\label{S:TheoremA}%S:ProofOfTheorem4&6&7
%We will prove Theorem \ref{T:SplitImpliesConeSequence} first, and then Theorem \ref{T:Cone&SplitMET}, since proof of theorems can be partially  derived from the proof of previous ones.\\
%\subsubsection{Multiplicative Ergodic Theorem(Review)}\label{S:MET}%S:MET
We first state the Multiplicative Ergodic Theorem from \cite{LL}, which fits the setting of the system, as well as the assumptions, in Theorem A:
\begin{theorem}\label{T:MET}(Multiplicative Ergodic Theorem)
Assume that $A(\cdot):\O\to L(X, X)$ is strongly measurable, $A(\o)$
is injective almost everywhere and
\[ \log^+||A(\cdot)|| \in L^1(\Omega, \mathcal{F}, P),
\]
Then there exists a $\theta$-invariant subset $\tilde
\Omega\subset\Omega$ of full measure such that for each $\omega \in
\tilde \Omega$ only one of the following conditions holds

\begin{itemize}
\item[(I)]$\l_0(\o)=\k(\o)$.

\vskip0.1in \item[(II)] There exist $k(\omega)$ numbers
$\l_0(\o)=\l_1(\o)>\ldots>\l_{k(\o)}(\o)>\l_{k(\o)+1}=\k(\o)$  and a splitting
     \[{X}=E_1(\o)\oplus\cdots\oplus E_{k({\omega})}(\o)\oplus F(\o)\]
of finite dimensional linear subspaces $E_j(\omega)$ and infinite
dimensional linear subspace $F(\o)$ such that
     \begin{itemize}
     \item[1)] Invariance: $k(\theta \omega)=k(\o)$,
     $\l_i(\theta\o)=\l_i(\o)$, $A(\omega) E_j(\o)=E_j(\theta \omega)$ and
     $A(\o)F(\o)\subset F(\theta\o)$;
     \item[2)] Lyapunov Exponents: \[\lim_{n\to\pm\infty}\frac1n\log\|T^n(\omega)v\|
     =\l_j(\o)\text{ for all }v\in E_j(\o)\setminus \{0\},1\le j\le k;\]
     \item[3)] Exponential Decay Rate on $F(\omega)$:
\[\limsup_{n\to+\infty}\frac1n\log\|T^n(\omega)|_{F(\o)}\|\le
     \k(\o)\] and $\text{if }v\in F(\o)\setminus\{0\}\text{ and
     }(T^n(\theta^{-n}\omega))^{-1} v\text{ exists for all }n\ge0$, which
is denoted by $T^{-n}(\omega)v$,
     then
     \[\liminf_{n\to+\infty}\frac1n\log\|T^{-n}(\omega)v\|\ge-\k(\o);\]

     \item[4)] Tempered Projections: The projection operators associated with the decompositions
     \[X=\Big(\bigoplus_{i=1}^j
E_i(\o)\Big)\oplus\Big(\big(\bigoplus_{i=j+1}^{k(\o)}E_i(\o)\big)\oplus
F(\o)\Big)=\Big(\bigoplus_{i=1}^{k(\o)}E_i(\o)\Big)\oplus F(\o)\]
are tempered;
\item[5)] Measurability:  $k(\o), \l_i(\o)$,
 and $E_j(\o)$ are measurable and the projection operators  are strongly measurable.
     \end{itemize}

     \vskip0.1in
\item[(III)] There exist infinitely many finite dimensional
subspaces $E_j(\o)$ and infinite dimensional subspaces $F_j(\o)$,
and infinitely many numbers

\[\l_0(\o)=\l_1(\o)>\l_2(\o)>\ldots>\k(\o) \; \text{with }\; \lim_{j\to+\infty}\l_j(\o)=\k(\o)
\]
such that
\begin{itemize}

\item[1)] Invariance: $\l_i(\theta\o)=\l_i(\o)$
$A(\o)E_j(\o)=E_j(\t\o)$, $A(\o)F_j(\o)\subset F_j(\t\o)$;
\item[2)] Invariant Splitting:
\[E_1(\o)\oplus\cdots\oplus E_j(\o)\oplus
      F_j(\o)={X}\; \text{ and }\; F_j(\o)=E_{j+1}(\o)\oplus
      F_{j+1}(\o);\]
\item[3)] Lyapunov Exponents:
      \[\lim_{n\to\pm\infty}\frac1n\log\|T^n(\o)v\|=\l_j(\o), \text {for all }
      v(\neq0)\in E_j(\o);\]
\item[4)] Exponential Decay Rate on $F_j(\o)$:
\[\lim_{n\to+\infty}\frac1n\log\|T^n(\o)\big|_{F_j(\o)}\|=\l_{j+1}(\o)
\] and if for $v\in F_j(\o)\setminus\{0\}$ such that $T^{-n}(\o)v$ exists for
all $n\ge0$, then
\[\liminf_{n\to+\infty}\frac1n\log\|T^{-n}(\omega)v\|\ge-\l_{j+1}(\o);
\]
\item[5)] Tempered Projections: The projection operators associated with the decomposition
          \[X=\left(\bigoplus_{i=1}^j E_i(\o)\right)\oplus
          F_j(\o)\]
          are tempered.
\item[6)] $\l_j(\o)$ and $ E_j(\o)$ are measurable and the
projection operators are strongly measurable.
\end{itemize}
\end{itemize}
\end{theorem}

%\subsubsection{Proof of Theorem \ref{T:SplitImpliesConeSequence}}\label{S:Theorem6}%S:Theorem6
For sake of simpleness, we hereafter will only prove Theorem A in the case that $\t$  is {\bf ergodic}. So the functions $\l_0(\cdot), \k(\cdot),k(\cdot),\l_i(\cdot)'s$ defined in Theorem \ref{T:MET} are now all constants. Let us denote them by $\l_0, \k,k,\l_i$, respectively. For the non-ergodic case, the proof will be the same except that all the constants we derived become $\t$-invariant functions.

By the assumption $\l_0>\k$ in Theorem A, only Case (II) (while $k<\infty$) and Case (III) (while $k=\infty$) in Theorem \ref{T:MET} can happen. Based on this, we introduce the so called {\it Lyapunov Norm} as follows. Fix a sequence of small numbers $\{\e_i\}_{1\le i<k+1}$ such that $\e_1\ll\l_1-\l_2$ and $0<\e_i\ll\min\{\l_{i}-\l_{i+1},\l_{i-1}-\l_i\}$ for $1<i<k+1$. For $1\le i<k+1$ and $\o\in\O$, define the norm $|\cdot|_{\o,i}$:
\begin{equation}\label{E:LyapunovNorm}%E:LyapunovNorm
|v|_{\o,i}=
\begin{cases}
\sum_{n=-\infty}^\infty\frac{|T^n(\o)v|}{e^{n\l_j+|n|\e_j}}, &\text{ when }v\in E_j(\o),\ 1\le j\le i,\\
\sum_{n=0}^\infty\frac{|T^n(\o)v|}{e^{n(\l_{i+1}+\e_{i+1})}}, &\text{ when }v\in E'_i(\o),\\
\sum_{j=1}^{i+1}|\pi_j^i(\o)v|_{\o,i}, &\text{ otherwise,}
\end{cases}
\end{equation}
where
\begin{equation*}
E'_i(\o)=
\begin{cases}
\left(\bigoplus_{j=i+1}^{k}E_j(\o)\right)\oplus
F(\o), &\text{ if case (II) holds},\\
 F_i(\o), &\text{ if case (III) holds};
\end{cases}
\end{equation*}
while $\pi^i_j(\o):X\to E_j(\o),$ for $j=1,\cdots,i$, and $\pi^i_{i+1}(\o):X\to E_{i}'(\o)$ are the projections associated with the splitting
$$X=\left(\bigoplus_{j=1}^iE_j(\o)\right)\bigoplus E'_i(\o).$$
By Theorem \ref{T:MET}, for each $1\le i<k+1$, $|\cdot|_{\o,i}$ is well defined $\mb P$-a.e. and is strongly measurable. The following lemma implies the relation between $|\cdot|_{\o,i}$ and the usual norm $|\cdot|$ in $X$.
\begin{lemma}\label{L:LyapunovMetric}%L:LyapunovMetric
For each $1\le i<k+1$, there exists a tempered function $K_i:\O\to [1,\infty)$ such that
\begin{itemize}
\item[i)] $|v|\le|v|_{\o,i}\le K_i(\o)|v|$;
\item[ii)] For all $v\in E_j(\o)$ and $1\le j\le i$,
$$e^{\l_j-\e_j}|v|_{\o,i}\le |A(\o)v|_{\t\o,i}\le e^{\l_j+\e_j}|v|_{\o,i};$$
\item[iii)] For all $v\in E'_i(\o)$
$$|A(\o)v|_{\t\o,i}\le e^{\l_{i+1}+\e_{i+1}}|v|_{\o,i}.$$
\end{itemize}
\end{lemma}
\begin{proof}
i) Clearly, $|v|\le |v|_{\o,i}$. As for the existence of the function $K_i$, it follows from Theorem \ref{T:MET} that there exist tempered functions $K_j':\O\to [1,\infty),\ 1\le j\le i+1$ such that, for $\mb P$-a.e. $\o\in\O$,
$$|T^n(\o)v|\le K_j'(\o) e^{n\l_j+\frac12|n|\e_j}|v|,\ \text{ for any } 1\le j\le i,\ v\in E_j(\o),\  n\in\mb Z,$$
and
$$|T^n(\o)v|\le K'_{i+1}(\o)e^{n(\l_{i+1}+\frac12\e_{i+1})}|v|,\ \text{ for any } v\in E'_i(\o),\ n\in \mb N\cup\{0\}.$$
(For more details, we refer the reader to \cite{LL}.) Moreover, by Theorem \ref{T:MET}, $\{\|\pi_j\|\}_{1\le j\le i+1}$ are tempered functions. So, by letting
$$K_i(\cdot)=(i+1)\cdot\max_{1\le j\le i+1}\left\{\frac{2K'_j(\cdot)\|\pi_j^i(\cdot)\|}{1-e^{-\frac12\e_j}}\right\},$$
we have completed the proof of i).

ii)-iii) follow from the definition of $|\cdot|_{\o,i}$ by straightforward computation, which we omit here.
\end{proof}

Now we are ready to construct the invariant cones in Theorem A. Given $l>0$, for $1\le i<k+1$ and $\o\in\O$, define
\vskip 1mm
\begin{equation}\label{E:ConeI}%E:ConeI
 C_i^l(\o)=\left\{v\in X\big|\ |\pi^i_{i+1}(\o)v|_{\o,i}\le l|v-\pi^i_{i+1}(\o)v|_{\o,i}\right\}.
\end{equation}
The properties of $C_i^l(\o)$ will be summarized in the following two lemmas
\vskip 2mm
\begin{lemma}\label{L:InvariantCi}%L:InvariantCi
For each $l>0$ and $1\le i<k+1$, there exists a tempered function $\chi^l_i(\o)\ge 1$ such that
$$\udi(A(\o)C^l_i(\o),X\setminus C^{le^{-\e_i}}_i(\t\o))\ge \frac1{\chi^l_i(\o)},\ \mb P-a.e..$$
\end{lemma}
\begin{proof}
By  ii)-iii) of Lemma \ref{L:LyapunovMetric}, we obtain that for all $v\in C^l_i(\o)$,
\begin{equation}\label{E:uniform-cones-contr}
\left|\pi_{i+1}^i(\t\o)A(\o)v\right|_{\o,i}\le le^{-\l_i+\l_{i+1}+\e_i+\e_{i+1}}\left|A(\o)v-\pi_{i+1}^i(\t\o)A(\o)v\right|_{\o,i}.
\end{equation}
Thus $A(\o)C_i^l(\o)\subset C_i^{le^{-\e_i}}(\t\o)$. Now, for any nonzero vector $w_1\in A(\o)C_i^l(\o)$ and any $w_2\in X\setminus C^{le^{-\e_i}}_i(\t\o)$, we define
$$\D=\max\left\{\left|\pi^i_j(\t\o)(w_1-w_2)\right|_{\t\o,i},\ 1\le j\le i+1\right\}.$$
By virtue of \eqref{E:uniform-cones-contr} (with $w_1=A(\o)v$ for some $v\in C_i^l(\o)$), one has
\begin{align*}
&l^{-1}e^{\l_i-\l_{i+1}-\e_i-\e_{i+1}}|\pi^i_{i+1}(\t\o)w_1|_{\t\o,i}-i\D\\
\le&\sum_{j=1}^i|\pi^i_j(\t\o)w_1|_{\t\o,i}-i\D\le \sum_{j=1}^i|\pi^i_j(\t\o)w_2|_{\t\o,i}\\
\le&l^{-1}e^{\e_i}|\pi^i_{i+1}(\t\o)w_2|_{\t\o,i}\le l^{-1}e^{\e_i}(|\pi^i_{i+1}(\t\o)w_1|_{\t\o,i}+\D),
\end{align*}
and hence,
\begin{equation}\label{E:Delta-lowerbdd}
\D\ge \frac{l^{-1}(e^{\l_i-\l_{i+1}-\e_i-\e_{i+1}}-e^{\e_i})|\pi^i_{i+1}(\t\o)w_1|_{\t\o,i}}{i+l^{-1}e^{\e_i}}.
\end{equation}
Due to the definition of $\D$, we see that $|w_2|_{\t\o,i}\ge |w_1|_{\t\o,i}-(i+1)\D$. So
\begin{align*}
|\pi^i_{i+1}(\t\o)w_1|_{\t\o,i}\ge& |\pi^i_{i+1}(\t\o)w_2|_{\t\o,i}-\D\\
\ge&(l^{-1}e^{\e_i}+1)^{-1}|w_2|_{\t\o,i}-\D\\
\ge& (l^{-1}e^{\e_i}+1)^{-1}|w_1|_{\t\o,i}-\left((i+1)(l^{-1}e^{\e_i}+1)^{-1}+1\right)\D.
\end{align*}
Therefore, together with \eqref{E:Delta-lowerbdd}, we get
\begin{equation}\label{Delta-lowerbdd-1}
\D\ge|w_1|_{\t\o,i}\cdot\frac{l(e^{\l_i-\l_{i+1}-\e_i-\e_{i+1}}-e^{\e_i})}
{l(e^{\l_i-\l_{i+1}-\e_i-\e_j}-e^{\e_i})+(li+e^{\e_i})(e^{\e_i}+l)}.
\end{equation}
On the other hand, by applying i) of Lemma \ref{L:LyapunovMetric}, one has
\begin{equation}\label{Delta-upperbdd-1}\D\le K_i(\o)\cdot |w_1-w_2|\cdot \max_{1\le j\le i+1}\{\|\pi_j^i(\t\o)\|\}.
\end{equation}
Combing \eqref{Delta-lowerbdd-1} with \eqref{Delta-upperbdd-1}, we obtain
\begin{equation*}\label{E:SeparationFormula}%E:SeparationFormula
\di\left(\frac{w_1}{|w_1|},X\setminus C^{le^{-\e_i}}_i(\t\o)\right)
\ge \frac{1}{\chi_i^l(\o)},
%\frac{l(e^{\l_i-\l_{i+1}-\e_i-\e_{i+1}}-e^{\e_i})\left(K_i(\t\o)\max_{1\le j\le i+1}\{\|\pi_j^i(\t\o)\|\right)^{-1}}
%{l(e^{\l_i-\l_{i+1}-\e_i-\e_{i+1}}-e^{\e_i})+(li+e^{\e_i})(e^{\e_i}+l)}.
\end{equation*}
where
\begin{equation}\label{E:SeparationFormula1}%E:SeparationFormula1
\chi_i^l(\o)=\left(1+\frac{(li+e^{\e_i})(e^{\e_i}+l)}
{l(e^{\l_i-\l_{i+1}-\e_i-\e_{i+1}}-e^{\e_i})}\right)K_i(\t\o)\max_{1\le j\le i+1}\{\|\pi_j^i(\t\o)\|\}.
\end{equation}
By the arbitrariness of $w_1$, we have proved the lemma.
\end{proof}

\begin{lemma}\label{L:CiSubCi+1}%L:CiSubCi+1
If $k>1$, then there exist positive tempered functions $\{l_i(\cdot)\}_{1\le i<k+1}$ with $l_i(\o)e^{-\e_i}\le l_i(\t\o)\le l_i(\o)e^{\e_i}$ for $\mb P$-a.e. $\o$, such that
 $$C^{l_i(\o)}_i(\o)\subset C^{l_{i+1}(\o)}_{i+1}(\o)\,\,\text{ for any }1\le i<k.$$
\end{lemma}
\begin{proof}
We will construct $l_i$ by induction. To start the procedure, we take $l_1\equiv1$. Suppose that $l_i$ is well defined for some $i\in [1,k)$, we now construct $l_{i+1}$.

 For $1\le i<k$, it is clear that
\begin{equation}\label{E:projectopn-j-j+1}
|\pi^{i}_{j}(\o)v|_{\o,i}=|\pi^{i+1}_{j}(\o)v|_{\o,i+1},\,\,\text{ for any }1\le j\le i\text{ and }v\in X.
\end{equation}
Now for any given $v\in C^{l_i(\o)}_i(\o)$, i) of Lemma \ref{L:LyapunovMetric} implies that
\begin{align*}\begin{split}%E:li+1
|\pi^{i+1}_{i+2}(\o)v|_{\o,i+1}&\le K_{i+1}(\o)|\pi^{i+1}_{i+2}(\o)v|= K_{i+1}(\o)|\pi^{i+1}_{i+2}(\o)\pi^{i}_{i+1}(\o)v|\\
&\le K_{i+1}(\o)|\pi^{i+1}_{i+2}(\o)\pi^{i}_{i+1}(\o)v|_{\o,i}\\
&\le K_{i+1}(\o)K_i(\o)\|\pi^{i+1}_{i+2}(\o)\|\cdot |\pi^{i}_{i+1}(\o)v|\\
&\le K_{i+1}(\o)K_i(\o)\|\pi^{i+1}_{i+2}(\o)\|l_i(\o)\sum_{j=1}^i|\pi^i_j(\o)v|_{\o,i}.
\end{split}
\end{align*}
Then \eqref{E:projectopn-j-j+1} entails that
\begin{equation}\label{E:li+1}
|\pi^{i+1}_{i+2}(\o)v|_{\o,i+1}\le K_{i+1}(\o)K_i(\o)\|\pi^{i+1}_{i+2}(\o)\|l_i(\o)\sum_{j=1}^{i+1}|\pi^{i+1}_j(\o)v|_{\o,i+1}.
\end{equation}
Since $K_{i+1},K_i,\|\pi^{i+1}_{i+2}\|$ and $l_i(\o)$ are tempered,  Lemma \ref{L:TemperedFunctionB} in the appendix implies that there is some tempered function $l_{i+1}(\o)\ge K_{i+1}(\o)K_i(\o)\|\pi^{i+1}_{i+2}(\o)\|l_i(\o)$ such that
$$e^{-\e_{i+1}}l_{i+1}(\o)\le l_{i+1}(\t\o)\le e^{\e_{i+1}}l_{i+1}(\o).$$
 Together with (\ref{E:li+1}), we have obtained that $C^{l_i(\o)}_i(\o)\subset C^{l_{i+1}(\o)}_{i+1}(\o)$, which completes the proof.
\end{proof}

Now we are ready to finish the proof of Theorem A:

\begin{proof}[Proof of Theorem A.]
When $k=1$, we define $C_1(\o)=C^l_1(\o)$ and $E_0(\o)=E_1(\o)$, $F_0(\o)=F(\o)$ for $\mb P$-a.e. $\o$ and some number $l>0$. Then, by Lemma \ref{L:InvariantCi} and Theorem \ref{T:MET}, this family of cones is a measurably contracting cone family.

When $1<k\le +\infty$, for each $1\le i<k+1$, we set $C_i(\o)=C^{l_i(\o)}_i(\o)$, where $l_i(\o)$ is defined in Lemma \ref{L:CiSubCi+1}. Also note that if one replace $l$ in (\ref{E:SeparationFormula1}) by $l_i(\o)$, then the resulted $\chi^{l_i(\o)}_i(\o)$ is still a tempered function. Then it follows from Lemmas \ref{L:InvariantCi}-\ref{L:CiSubCi+1} that $\{C_i(\cdot)\}_{\o\in\O}$ forms a family of nested measurably cones satisfying \eqref{Strong-focusing-1} and \eqref{E:ConeI-1}.
Thus, we have completed the proof of Theorem A.
\end{proof}

\begin{remark}\label{R:FiniteConeFamily}%R:FiniteConeFamily
Lemma \ref{L:CiSubCi+1} is mainly dealing with the case that $k=\infty$ and allow us to construct infinitely many nested co-invariant cones. When $k<\infty$ or if one only expect to construct finitely many monotonic cone families (even though $k=\infty$), the proof will be much simpler. In the case that $k<+\infty$, we only need the Lyapunov norm $|\cdot|_{\o,k}$ with which we define, for any given $l\ge 1$,
$$C_i(\o)=\left\{v\in X\Big|\ \sum_{j=i+1}^{k+1}|\pi^k_j(\o)v|_{\o,k}\le l\sum_{j=1}^i|\pi^k_j(\o)v|_{\o,k}\right\},\ 1\le i\le k.$$
It is easy to check such cone families satisfy the conditions in Theorem A. When $k=\infty$, for any finite integer $m\ge 2$, one can use the exactly same idea to construct $m$-many nested co-invariant cones as in the case that $k<+\infty$.
\end{remark}
\vskip 2mm

\subsection{Proofs of Theorem \ref{T:SplitImpliesConeL1} and \ref{T:Cone&SplitMET}}\label{S:Theorem4}%S:Theorem4
%The idea to prove Theorem \ref{T:SplitImpliesConeL1} is to modify the original system to fit in the setting of Theorem \ref{T:SplitImpliesConeSequence}.  \\

\begin{proof}[Proof of Theorem \ref{T:SplitImpliesConeL1}]
For convenience, we modify the original system by defining $\tilde A(\o)=\frac{A(\o)}{\|A(\o)\|}$ for all $\o\in\O$, and denote $\tilde T$ the linear system generated by $\tilde A$. It is easy to see that such re-scaling to $\tilde{A}$ does not affect Conditions  D1),D2) and D3'). Let $\tilde\l_0,\tilde\k$ be defined in \eqref{E:principle-LE}-\eqref{E:ess-principle-LE} with $T$ being replaced by $\tilde{T}$. Then
\begin{equation}\label{E:alternative-LP-relation}
\tilde\l_0=0\,\,\text{ and }\,\,\tilde\k=\k-\l_0.
\end{equation}
By virtue of D4), we further have
$$\log^+\|(\tilde A|_E)^{-1}\|=\log^+(\|A\|\|(A|_E)^{-1}\|)\in \mc L^1(\O,\mc F,\mb P),$$ by which
the Kingman's subadditive ergodic theorem (Lemma \ref{L:Kingman} in the appendix) will imply that, for $\mb P$-a.e. $\o$,
\begin{equation}\label{E:BackRate}%E:BackRate
\tilde \l^{-}_{E}(\o)=:\lim_{n\to\infty}\frac1n\log\left\|\left(\tilde T^n(\o)|_{E(\o)}\right)^{-1}\right\| \text{ exists and is finite}.
\end{equation}
Recall that $\log^+ \|\tilde A\|=0.$ One can also define, by the Kingman's ergodic theorem, that
\begin{equation}\label{E:ERate}%E:ERate
\tilde\l_{E}(\o):=\lim_{n\to\infty}\frac1n\log\left\|\left(\tilde T^n(\o)|_{E(\o)}\right)\right\|,
\end{equation}
\begin{equation}\label{E:FRate}%E:FRate
\tilde\l_{F}(\o):=\lim_{n\to\infty}\frac1n\log\left\|\left(\tilde T^n(\o)|_{F(\o)}\right)\right\|,
\end{equation}
and
\begin{equation}\label{E:Kapa}%E:Kapa
\tilde\k_{F}(\o)=\lim_{n\to\infty}\frac1n\log\left\|\left(\tilde T^n(\o)|_{F(\o)}\right)\right\|_\k,
\end{equation}
which are all $\t$-invariant.  By definition, it is obvious that
\begin{equation}\label{E:ERate1}%E:ERate1
\tilde\l_E(\o)\ge \lim_{n\to\infty}\frac1n\log\left\|\left(\tilde T^n(\o)|_{E(\o)}\right)^{-1}\right\|^{-1}=-\tilde\l^{-}_{E}(\o)>-\infty.
\end{equation}
It then follows from Condition D3') and (\ref{E:FRate})-(\ref{E:ERate1}) that, for $\mb P$-a.e. $\o\in\O$,
\begin{align}\begin{split}\label{E:PositiveGap}%E:PositiveGap
\tilde\l_{E}(\o)&\ge-\tilde\l^{-}_{E}(\o)=-\lim_{j\to\infty}\frac1{n_{j}(\o)}\log\left\|\left(\tilde T^{n_{j}(\o)}(\o)|_{E(\o)}\right)^{-1}\right\| \\
&\ge \lim_{j\to\infty}\frac1{n_{j}(\o)}\log\left(\left(K(\o)\right)^{-1}e^{j\d(\o)}\left\|\left(\tilde T^n(\o)|_{F(\o)}\right)\right\|\right)\\
&\ge \liminf_{j\to\infty}\frac{j}{n_{j}(\o)}\d(\o)+\tilde\l_{F}(\o).
\end{split}
\end{align}
Now we apply Theorem \ref{T:MET} to system $(\O,\mc F,\mb P,\t,\tilde T,X)$. By the invariance of splitting $X=E(\cdot)\oplus F(\cdot)$ and the definitions of $\tilde\l^{-}_{E}$ and $\tilde\l_{F}$, it is easy to see that the Lyapunov exponents corresponding to vectors in $E(\cdot)$ is bounded below by $-\tilde\l^{-}_{E}(\cdot)$ and  the Lyapunov exponents corresponding to vectors in $F(\cdot)$ is bounded above by $\tilde\l_{F}(\cdot)$. As a consequence, (\ref{E:PositiveGap}) entails that
$$\tilde\l_{0}=\tilde\l_{E}\ge -\tilde\l^{-}_{E}>\tilde\l_{F}\ge \tilde\k_{F}=\tilde\k,$$ and hence,  $\l_{0}>\k$ by \eqref{E:alternative-LP-relation}.

Moreover, we have that for some $\t$-invariant function $l:\O\to \mb N$
$$E(\o)=\bigoplus_{i=1}^{l(\o)}E_i(\o)\text{ and }\l_{l(\o)}\ge \l_{l(\o)+1}+\liminf_{j\to\infty}\frac{j}{n_{j}(\o)}\d(\o).$$
Then the desired measurably contracting cone $C(\o)$ can be derived from Lemma \ref{L:InvariantCi}. The proof of Theorem \ref{T:SplitImpliesConeL1} has been completed.
\end{proof}

\vskip 2mm
\begin{proof}[Proof of Theorem \ref{T:Cone&SplitMET}]\label{S:Theorem7}%S:Theorem7
We first note that: a) $\Longrightarrow$ b) by Theorem \ref{T:ConeImpliesSplit};
b) $\Longrightarrow$ c) by Theorem \ref{T:SplitImpliesConeW};
c) $\Longrightarrow$ d) by Theorem \ref{T:ConeImpliesSplitW};
e) $\Longrightarrow$ a) by Theorem A.
 So, in order to complete the proof, it suffices to show that d) $\Longrightarrow$ e).

Let $(\O,\mc F,\mb P,\t,T,X)$ be the system which possesses a Measurably Dominated Splitting $E(\cdot)\oplus F(\cdot)$ in Probability satisfying D1),(D2) and D3'). Since $\log^{+}\|A\|\in\mc L^{1}(\O,\mc F,\mb P)$, it follows from the Kingman's subadditive ergodic theorem that, for $\mb P$-a.e. $\o$, one can define $\l_{E},\l_{F}$ and $\k_{F}$ as in (\ref{E:ERate})-(\ref{E:Kapa}) (with $\tilde{T}$ replaced by $T$), respectively.

Now, let
$$E'(\o)=\left\{v\in E(\o)\Big|\ \lim_{n\to\infty}\frac1n\log|T^{-n}(\o)v|<\infty\right\}.$$
Clearly, $E'(\o)\subset E(\o)$. By Theorem \ref{T:MET} and $\l_{0}>-\infty$, one can obtain another invariant splitting $X=E'(\o)\oplus F'(\o)$ such that $\dim E'(\cdot)\ge 1$ and $F'(\o)\supset F(\o)$. Moreover, for $\mb P$-a.e. $\o$,
\begin{equation}\label{E:BackRate-T}%E:BackRate
\l^{-}_{E'}(\o)=:\lim_{n\to\infty}\frac1n\log\left\|\left(T^n(\o)|_{E'(\o)}\right)^{-1}\right\| \text{ is well-defined and finite.}
\end{equation}
By the definitions of $\l^{-}_{E'}$, the Lyapunov exponents corresponding to vectors in $E'(\cdot)$ is bounded below by $-\l^{-}_{E'}(\cdot)$ and  the Lyapunov exponents corresponding to vectors in $F'(\cdot)$ is bounded above by $\max\{\l_{F}(\cdot),-\infty\}$.

By virtue of D3'), we can similarly employ \eqref{E:PositiveGap}, with $E$ (resp. $\tilde{T}$) being replaced by $E'$ (resp. $T$), to obtain that
$$\l_{E'}\ge -\l^{-}_{E'}>\l_{F}.$$
Note also that $\l_E\ge-\l^{-}_{E'}$ and $\l^{-}_{E'}$ is finite. It entails that
$$\l_{0}=\l_E\ge \l_{E'}\ge-\l^{-}_{E'}>\max\{\l_{F},-\infty\}\ge \k.$$
Thus, we have completed the proof of ``d) $\Longrightarrow$ e)".
\end{proof}
%which allow us to apply Lemma \ref{L:InvariantCi} and \ref{L:CiSubCi+1} to derive the measurably contracting cone family.
\begin{remark}\label{R:-infty}%R:-infty
When $E'(\o)\subsetneq E(\o)$, it falls into case (II) of  Theorem \ref{T:MET}. In this case, $\l_{F}=-\infty$ (and hence, $\k=-\infty$), $\dim E'(\o)<\dim E(\o)$ and $F(\o)\subsetneq F'(\o)$. Thus from the measurably dominated splitting in probability, one can only derive a measurably contracting cone families with lower dimensions.
\end{remark}

%%%%%%%%%%%%%%%%%%%%%%% New Subsection %%%%%%%%%%%%%%%%%%%%%%%%%%%

%%%%%%%%%%%%%%%%%%%%%%% Appendix %%%%%%%%%%%%%%%%%%%%%%%%%%%%%%%%%%%%%
\appendix
\section{Ergodic Theory}\label{S:AppendixA}%S:AppendixA
In this appendix, we state several known results. The following two lemmas are standard results in ergodic theory.
\begin{lemma}\label{L:TemperedFunctionA}%L:TemperedFunctionA
Let $(\Lambda,\mc G, \mu)$ be a probability space and $\vartheta:\Lambda\to\Lambda$ be a $\mu$-measure preserving transformation. The for any measurable function $f:\Lambda\to\mb R$ we have that
$$\liminf_{n\to\infty}\frac1n f(\vartheta^{n}(x))\le 0\ a.s..$$
Moreover, if there exists $F\in L^{1}(\mu)$ such that
$$f(\vartheta(x))-f(x)\le F(x)\ a.s.,$$
or there exists $G\in L^{1}(\mu)$ such that
$$f(\vartheta(x))-f(x)\ge G(x)\ a.s.,$$
then $$f\circ \vartheta-f \in L^{1}(\mu)$$
and $$\lim_{n\to\infty}\frac1nf(\vartheta^{n}(x))=0.$$
\end{lemma}

\begin{lemma}\label{L:TemperedFunctionB}%L:TemperedFunctionB
Let $f:\O\to (0,+\infty)$ be tempered and $\g:\O\to(0,+\infty)$ be a $\t$-invariant random variable. Then, there is a tempered random variable $R:\O\to(0,+\infty)$ such that for any $\o\in \O$
\begin{itemize}
\item[i)] $\frac1{R(\o)}\le f(\o)\le R(\o)$;
\item[ii)] $e^{-\g(\o)|n|}R(\o)\le R(\t^{n}\o)\le e^{\g(\o)|n|}R(\o)$.
\end{itemize}
\end{lemma}

\begin{lemma}\label{L:Kingman}%L:Kingman
{\rm(Kingman's Subadditive Ergodic Theorem).} $(M,\Sigma,\r)$ denotes a
probability space, and $f:\ M\to M $ a measurable map preserving
$\r$. Let $\{F_n\}_{n>0}$ be a sequence of measurable functions from
$M$ to $\mb{R}\bigcup \{-\infty\}$ satisfying the conditions:
\begin{align*}
&(a)\ integrability:\ F_1^+\in L^1(M,\Sigma,\r),\\ &(b)\
subadditivity:\ F_{m+n}\leq F_m+F_n\circ f^m almost\ everywhere.
\end{align*}
Then there exists an $f$-invariant measurable function $F:\
M\to\mb{R}\bigcup \{-\infty\}$ such that $F^+\in L^1(M,\S,\r)$,
$$\ds \lim_{n\to\infty}\frac1n\, F_n=F\ a.e.$$ and
$$\lim_{n\to\infty} \frac1n \int F_n(x)\r(dx)= \inf_{n\in \mathbb{N}}\frac1n
\int F_n(x)\r(dx)= \int F(x)\r(dx).$$
\end{lemma}
\bigskip

\section{Strong normal co-invariant cone-family}\label{S:AppB}%S:AppB
In this appendix, we construct a cone family $\mc C'$ satisfying C1)-C4) based on a given cone family $\mc C$ satisfying C1)-C3). We define that for any $\o\in \O$
\begin{itemize}
\item[a)] $E_0'(\o)=A(\t^{-1}\o)E_0(\t^{-1}\o)$ and $F_0'(\o)=F_0(\o)$;
\item[b)] $C'(\o)=\left\{kv\big|\ |v|=1,k\in \mb R, \inf_{u\in A(\t^{-1}\o)C(\t^{-1}\o)\cap S}|v-u|\le \frac1{4\chi(\t^{-1}\o)}\right\}$.
\end{itemize}
From the definition of $\mc C'$, we have that for $\mb P$-a.e. $\o$,
$$A(\t^{-1}\o)C(\t^{-1}\o)\subset C'(\o)\subset  C(\o).$$
Thus $\mc C'$ satisfies condition C1).\\

By C3), and noting that for any $v,u\in S$, $$\udi(v,\spa\{u\})\le |v-u|\le 2\udi(v,\spa\{u\}),$$
then we have the following
\begin{align*}
&\inf_{u\in A(\o)C(\o)\cap S} \udi(u,X\setminus C'(\t\o))\\
\ge& \frac12 \inf\left\{|v-u|\big|\ u\in A(\o)C(\o)\cap S,v\in (X\setminus C'(\t\o))\cap S\right\}\\
\ge& \frac1{8\chi(\o)},
\end{align*}
and
\begin{align*}
&\udi(F_0(\o),C'(\o))\\
\ge&\frac12\inf\left\{|v-u|\big|\ v\in F_0(\o)\cap S, u\in C'(\o)\cap S\right\}\\
\ge &\frac12\left(\udi(A(\o)C(\o),X\setminus C(\o))-\sup\left\{|v-u|\big|\ v\in A(\o)C(\o)\cap S, u\in C'(\t\o)\cap S\right\}\right)\\
\ge& \frac3{8\chi(\o)}.
\end{align*}
So by taking $\chi'=\frac18\chi$, we have shown that $\mc C'$ satisfies C3) and C4). It remains to show C2).\\

By the strong measurability of $A:\O\to L(X)$ and separability of $X$, $A(\t^{-1}\cdot)E_0(\t^{-1}\cdot)$ is measurable. Actually, there exist random variables in $X$,
$\{v_i(\cdot)\}_{1\le i\le \dim E_0}$, such that $E_0(\cdot)=\spa\{v_i(\cdot)\}_{1\le i\le \dim E_0}$ and
$$|v_i(\cdot)|=1, \udi(v_i(\cdot),\spa\{v_j(\cdot)\}_{j\neq i})>\frac12,\ \forall 1\le i\le \dim E_0.$$
For details we refer the reader to \cite{LL}. Then by separability of $X$, $v_i(\cdot)$ can be approached by a sequence of simple functions $\{v^n_i(\cdot)\}_{n\ge 1}$ (piecewisely constant function) whose image under $A(\cdot)$ is also a measurable function by strong measurability of $A$. For large enough $n$, $\dim\spa\{A(\cdot)v^n_i(\cdot)\}=\dim E_0$, and $E^n(\cdot):= \spa\{A(\t^{-1}\cdot)v^n_i(\t^{-1}\cdot)\}$ is measurable (also refer to \cite{LL}), which converges to $E'_0(\cdot)$ as $n\to \infty$, thus $E'_0$ is measurable.\\
The next we show that the projections, $\pi_{E'_0}'(\cdot),\pi_{F_0}'(\cdot)$, associated to splitting $X=E'(\cdot) \oplus F_0(\cdot)$ are strongly measurable. Let $\{v'_i(\cdot)\}_{1\le i\le \dim E_0}$ be measurable unit basis of $E'_0(\cdot)$, and $\pi_{E_0}(\cdot),\pi_{F_0}(\cdot)$ be projections associated to splitting $X=E_0(\cdot)\oplus F_0(\cdot)$. Since, for $\mb P$-a.e. $\o$, $E'_0(\o)\cap F_0(\o)=0$, $\left\{\frac{\pi_{E_0}(\o)v'_i(\o)}{|\pi_{E_0}(\o)v'_i(\o)|}\right\}_{1\le i\le \dim E_0}$ forms a unit base of $E_0(\o)$ and varies measurably on $\o$. For any $v\in X$ and $\mb P$-a.e. $\o$, we have that
\begin{align}\begin{split}\label{E:ProjectionE'}%E:ProjectionE'
v&=\pi_{E_0}(\o)v+\pi_{F_0}(\o)v\\
&=\sum_{i=1}^{\dim E_0} a_i^\o(\pi_{E_0}(\o)v)\frac{\pi_{E_0}(\o)v'_i(\o)}{|\pi_{E_0}(\o)v'_i(\o)|}+\pi_{F_0}(\o)v\\
&=\sum_{i=1}^{\dim E_0}\frac{a_i^\o(\pi_{E_0}(\o)v)}{|\pi_{E_0}(\o)v'_i(\o)|}v_i'(\o)-\sum_{i=1}^{\dim E_0} \frac{a_i^\o(\pi_{E_0}(\o)v)}{|\pi_{E_0}(\o)v'_i(\o)|}\pi_{F_0}(\o)v_i'(\o)+\pi_{F_0}(\o)v,
\end{split}
\end{align}
where $\{a_i^\o(\pi_{E_0}(\o)v)\}_{1\le i\le \dim E_0}\subset \mb R$ are the unique coordinates of $\pi_{E_0}(\o)v$ under the base  $\left\{\frac{\pi_{E_0}(\o)v'_i(\o)}{|\pi_{E_0}(\o)v'_i(\o)|}\right\}_{1\le i\le \dim E_0}$. Each $a_i^\o$ can be viewed as a bounded linear functional defined on $E'_0(\o)$. By Lemma 7.5 in \cite{LL}, each $a_i^\o$ can be extended to a functional in $X^*$ and varies strongly measurable on $\o$. We also denote the extended functional by $a_i^\o$ without making any confusion. By (\ref{E:ProjectionE'}), we have that
$$\pi_{E_0'}'(\cdot)=\sum_{i=1}^{\dim E_0}\frac{a_i^{\cdot}\circ \pi_{E_0}(\cdot)}{|\pi_{E_0}(\cdot)v'_i(\cdot)|},$$
which is strongly measurable by the separability of $X$ and strong measurability of $\pi_{E_0}(\cdot)$ and measurability of $v_i'(\cdot)$.

 \addcontentsline{toc}{section}{References}

\bibliographystyle{plain}

\end{document}